\documentclass[12pt]{amsart}
\usepackage{amsmath, amsfonts, amssymb,comment, hyperref, graphicx,psfrag,float,amsthm,color}
\usepackage{diagbox, multicol}
\newcommand{\Mod}[1]{\ (\mathrm{mod}\ #1)}

\newtheorem{theorem}{Theorem}[section]
\newtheorem{proposition}[theorem]{Proposition}
\newtheorem{lemma}[theorem]{Lemma}

\newtheorem{corollary}[theorem]{Corollary}

\usepackage{chngcntr}
\usepackage{caption}
\usepackage{subcaption}

\counterwithin{table}{section}

\begin{document}

\title[On the nonorientable 4-genus of double twist knots]{On the nonorientable 4-genus\\ of double twist knots}
\author{Jim Hoste}
\email{jim\_hoste@pitzer.edu}  
\author{Patrick D.~Shanahan}
\email{patrick.shanahan@lmu.edu}
\author{Cornelia A.~Van Cott}
\email{cvancott@usfca.edu}

\maketitle

\begin{abstract}
We investigate the nonorientable 4-genus $\gamma_4$ of a special family of 2-bridge knots, the double twist knots $C(m,n)$. Because the nonorientable 4-genus is bounded by the nonorientable 3-genus, it is known that $\gamma_4(C(m,n)) \le 3$. By using explicit constructions to obtain upper bounds on $\gamma_4$ and known obstructions derived from Donaldson's Diagonalization Theorem to obtain lower bounds on $\gamma_4$, we produce infinite subfamilies of $C(m,n)$ where $\gamma_4=0,1,2,$ and $3$, respectively. However, there remain infinitely many double twist knots where our work only shows that $\gamma_4$ lies in one of the sets $\{1,2\}, \{2,3\}$, or $\{1,2,3\}$. We tabulate our results for all $C(m,n)$ with $| m|$ and $|n|$ up to 50. We also provide an infinite number of examples which answer a conjecture of Murakami and Yasuhara. 
\end{abstract}



\section{Introduction}
The {\it crosscap} number, or {\it nonorientable genus},  of a knot $K$ in $S^3$, denoted here as $\gamma_3(K)$,  was introduced by Clark in \cite{Clark_1978}. It is defined as the smallest  first Betti number of any embedded, compact, connected, nonorientable surface in $S^3$ that spans $K$. For convenience, the crosscap number of the unknot is defined to be zero. Similarly, the {\it nonorientable 4-genus} of a knot $K$ in $S^3$, introduced in \cite{Murakami_Yasuhara} and denoted here as $\gamma_4(K)$, is
defined as the smallest  first Betti number of any smooth, properly embedded, compact, connected,  nonorientable surface in $B^4$ that spans $K$. Moreover, $\gamma_4(K)$ is defined to be zero if $K$ is a slice knot. Clearly, $\gamma_4(K)\le \gamma_3(K)$ for any knot $K$. Notice that we are requiring the surfaces in $B^4$ to be smooth (unless otherwise stated). Work of Gilmer and Livingston~\cite{GL} showed that the smooth and topological categories produce different results for $\gamma_4$. In particular, they showed there exist knots that bound topological but not smooth M\"obius bands in $B^4$.

In 1975, years before the invariant $\gamma_4$ was formally defined, Viro~\cite{Viro} proved that the figure eight knot, $4_1$,  cannot bound a smooth M\"obius band in the 4-ball. Hence $\gamma_4(4_1) \geq 2$. As the checkerboard shading of the standard diagram for the figure eight gives a nonorientable surface with first Betti number 2, it follows that $\gamma_4(4_1)=\gamma_3(4_1) = 2$. Subsequent work has filled in the nonorientable 4-genus of all prime knots up through 10 crossings~\cite{Ghanbarian2020, JK}. (Note  that \cite{Ghanbarian2020, JK}, as well as {\sl KnotInfo} \cite{knotinfo},   define $\gamma_4$ of a slice knot to be one rather than the original definition of zero. Furthermore, the invariant is called ``4D Crosscap Number'' in KnotInfo.) The invariant has been studied extensively in the special case of torus knots~\cite{allen2020, Batson, binns2021, feller2021, jabuka2020,  jabuka2019, Lobb, Longo2020}. 

In addition to torus knots, another natural family of knots to consider are the 2-bridge knots $K_{p/q}$. The figure eight knot is a two bridge knot, as are all $(2,2k+1)$-torus knots which are easily seen to satisfy $\gamma_4(T(2,2k+1)) = \gamma_3(T(2,2k+1))= 1$ (assuming $|2k+1|>1$). In general, however, the nonorientable four genus of 2-bridge knots is not known. 

A nice subset of 2-bridge knots are the double twist knots, which will be the focus of this paper. Recall that a 2-bridge knot is one with a diagram of the form shown in Figure~\ref{4-plat diagram}, which we denote as $C(a_1, a_2, \dots, a_n)$, as is done in \cite{Cromwell2004}.   
\begin{figure} 
   \centering
   \includegraphics[width=4in, angle=0]{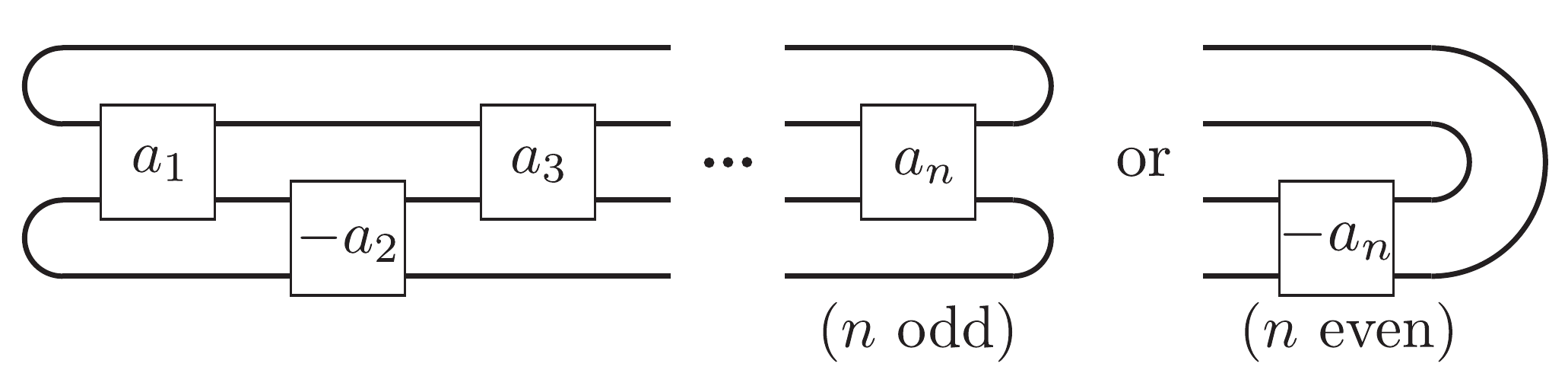}
   \caption{The diagram $C(a_1, a_2, \dots, a_n)$ of a 2-bridge knot.} 
\label{4-plat diagram}
\end{figure}
In the diagram, a box labeled $k$ denotes $k$ right-handed half twists between the two strands contained in the box. Note that $-k$  right-handed half twists are regarded as $k$ left-handed half twists. In particular, if $a_i>0$ for all $i$,  or $a_i<0$ for all $i$, then the diagram in Figure~\ref{4-plat diagram} is alternating and minimal in crossing number.  
The sequence $a_1, a_2, \dots, a_k$ gives rise to the reduced fraction $\frac{p}{q}$, via the (additive) continued fraction, 
$$\frac{p}{q}=[a_1, a_2, \dots, a_k]=a_1+\cfrac{1}{a_2+\cfrac{1}{\ddots +\cfrac{1}{a_k}}}.$$ 
This associates to $C(a_1, a_2, \dots, a_k)$ a rational number $p/q$, with $0<|q|<p$. In the case of a knot, $p$ is odd, while in the case of a 2-component link, $q$ will be odd and $p$ even. Clearly, we can work backwards from any such fraction to obtain a 2-bridge link, which we denote $K_{p/q}$. Moreover, two such fractions $p/q$ and $p'/q'$ represent ambient isotopic 2-bridge knots if and only if $p'=p$ and either  $q' \equiv q \text{ (mod $p$)}$ or $q'q\equiv 1 \text{ (mod $p$)}$. Because we will also consider a 2-bridge knot $K_{p/q}$ and its mirror image, $K_{-p/q}$,  as equivalent, the equivalence on fractions becomes $p'=p$ and either  $q' \equiv \pm q \text{ (mod $p$)}$ or $q'q\equiv \pm1 \text{ (mod $p$)}$.   Note that $\gamma_d(K_{p/q})=\gamma_d(K_{-p/q})$ for $d$ equal to either 3 or 4. 

 A {\it double twist knot} is a knot in $S^3$ with diagram $C(m,n)$,  where $m$ and $n$ are any integers, at least one of which is even. (If both are odd, $C(m,n)$ is a link of two components.) These include the {\it twist knots} $C(m,2)$ as well as the $(2, 2k+1)$-torus knots, $C(1,2k)$, which have $\gamma_4=1$ as mentioned above. In \cite{feller2021}, Feller and Golla determine $\gamma_4$ for an infinite family of twist knots.

The double twist knot $C(m,n)$ is a 2-bridge knot with classifying fraction
$$p/q=m+\frac{1}{n}=\frac{mn+1}{n}.$$ 
Thus, $C(m,n)$ is the unknot if and only if $mn+1=\pm 1$, or equivalently, if $mn=0$ or $mn=-2$. (If we assume that $m>1$, and $n$ is even and nonzero, then $C(m,n)$ is never the unknot.)  The classifying fraction also tells us that the 
determinant of $C(m,n)$ is $|mn+1|$. 
If $m$ and $n$ are both positive, then the diagram is alternating and the crossing number of $C(m,n)$ is $m+n$. If $m>0$ and $n<0$, then the diagram $C(m,n)$ is no longer alternating. However, we may change the diagram to $C(m-1,1,-n-1)$ which is alternating and represents the same knot because it has the same classifying fraction. Thus, when $m>0$ and $n<0$, the crossing number of $C(m,n)$ is $m-n-1$.

Throughout this paper, we consider two knots to be equivalent if there is a homeomorphism of $S^3$ taking one knot to the other.  Therefore  trading $m$ and $n$ or negating both $m$ and $n$ does not change the knot type.  By trading $m$ and $n$ or negating both, if necessary, we may assume that $m>0$ and $n$ is even. Because we know $\gamma_4=1$ when $m=1$ (with the exception $\gamma_4(C(1,-2))=0$),  throughout most of this paper  {\bf we will assume that  $\boldsymbol {m>1}$  and $\boldsymbol n$ is even and nonzero}.  Even with these restrictions, some knot types are still repeated. By ``flipping over the clasp'' in a twist knot, we see that $C(m,-2)=C(m-1,2)$. (In fact, this move is a flype.) Finally, if both $m$ and $n$ are positive and even, then both $C(m,n)$ and $C(n,m)$ meet the restriction of having the first parameter positive and the second one even. 

Using the work of Lisca \cite{Lisca} we first derive the following theorem which describes the double twist knots which are slice and hence have $\gamma_4=0$.

\newtheorem*{thm0}{Theorem \ref{double twist knots with gamma4 equal to 0}}
\begin{thm0}If  $m>1$ and $n$ is even and nonzero, then $\gamma_4(C(m,n))=0$  if and only if $|m-n|=2$ or $(m,n)=(5,-2)$. 
\end{thm0}

Our main results are the following theorems.

\newtheorem*{thm1}{Theorem \ref{double twist knots with gamma4 equal to 1}}
\begin{thm1}
If $m>0$, $n$ is even and nonzero,  and $(m,n)$ is one of the following, then $\gamma_4(C(m,n))=1$.  

\begin{enumerate}
\item $(1, n)$, except the unknot $(1,-2)$, 
\item  $(2,-6)$, $(2,-10)$, $(3,8)$, $(5,-6)$, $(6,-6)$, 
\item $(4,n)$, except $(4,2)$ and $(4,6)$,
\item $(m,\pm 4)$, except $(6,4)$, or
\item $(m,n)$ where $|m-n|\in \{1, 3, 6, 7\}$, except $(5,-2)$.
\end{enumerate}
\end{thm1}

\newtheorem*{thm2}{Theorem \ref{double twist knots with gamma4 equal to 2}}
\begin{thm2}
If $m>1$, $n$ is even and nonzero,  and $(m,n)$ is one of the following, then $\gamma_4(C(m,n))=2$.
\begin{enumerate}
\item $(2,n)$, with $n>0$ and $n\equiv 2\Mod{4}$,  

\item $(m,2)$, with $m\equiv 2 \Mod{4}$, 

\item $(m,n)$, with $m \equiv 3 \Mod{4}$, $n<0$, and $n\equiv 2 \Mod{4}$,    

\item  $(m,m)$, with $m \equiv 0 \Mod{4}$ and $m$ not a square, 

\item $m$ is odd, $n\equiv 0 \Mod{4}$, $n<0$,  and $n\ne -4$,  

\item $(8,n)$, with $n<0$ and $n\ne -4$,  

\item $(m,-8)$, with $m$ even and $m\ne 4$, 

\item $m \equiv 1\pmod{4}$, $m$ not a square, $n\equiv 2 \pmod{4}$, and $n>m+2$, 

\item $m \equiv 3 \pmod{4}$,  $n>0$, $n\equiv 2 \pmod{4}$, and $m>n+2$, 

\item $(m, m+10)$, with $m>2$  even, $m$ not a square, and  $m+10$ not a square, 

\item $(13+4k,8+4k)$ where $k\ge 0$, $13+4k$ not a square, and $8+4k$ not a square, 

\item $(7+4k,12+4k)$ where $k\ge 0$ and $12+4k$ not a square,

\item\label{twistex} $(2, -(18+100 k))$ where $k\ge 0$, 

\item $(13+36k,-10)$ where $k\ge0$. 

\end{enumerate}  
\end{thm2}

\newtheorem*{thm3}{Theorem \ref{double twist knots with gamma4 equal to 3}}
\begin{thm3}
If $(m,n)=(22+8k,62+8k)$ with $k\ge 0$, then $\gamma_4(C(m,n))=3$.
\end{thm3}

In \cite{Murakami_Yasuhara}, the authors prove that $\gamma_4(K)\le 2g_4(K)+1$ for all knots, where $g_4$ is the orientable 4-ball genus. They then conjecture that there is a knot where equality is achieved. The first known example is given in \cite{JK}. The double  twist knots in Theorem~\ref{double twist knots with gamma4 equal to 3}  provide infinitely many such examples (see Corollary~\ref{MY conjecture}). 

Restricting our attention to twist knots we also obtain the following which generalizes the results of Feller and Golla \cite{feller2021}.

\newtheorem*{thmtwistknots}{Theorem \ref{values of t and j that work}}
\begin{thmtwistknots}
For the following twist knots $K$, we have  $\gamma_4(K)=2$.  
\begin{enumerate}
\item $K=C(4 j+4 t k, 2)$, where $k\ge 0$   and $(t,j)$ are specified as in Table~\ref{tj-values for m congruent to 0 mod 4}.
\item $K=C(1+4 j+4 t k, 2)$, where $k\ge 0$   and $(t,j)$ are specified as in Table~\ref{tj-values for m congruent to 1 mod 4}.
\end{enumerate}
\end{thmtwistknots}

While these results determine $\gamma_4$ for infinitely many twist and double twist knots and significantly extend the known values of $\gamma_4$ for 2-bridge knots, 
 there still remain infinitely many $C(m,n)$ for which we can only determine that  $\gamma_4$ lies in one of the sets $\{1,2\}, \{2,3\}$, or $\{1,2,3\}$. In Tables~\ref{table1} and \ref{table2}, we record our results for all  $C(m,n)$ with $1\le m \le 50$ and $-50 \le n \le 50$.

The paper proceeds as follows. In Section~\ref{general methods section}, we describe the methods used determine upper and lower bounds for $\gamma_4(K)$. These methods are due to many authors and apply to any knot. In Section~\ref{Properties of Double Twist Knots}, we compute various invariants of double twist knots that will be needed to apply the methods described in Section~\ref{general methods section}. In Section~\ref{gamma_4 is 0 or 1}, we determine precisely which double twist knots have $\gamma_4=0$ and give a number of infinite families of double twist knots for which $\gamma_4=1$. In Sections~\ref{gamma_4 is 2} and \ref{gamma_4 is 3}, we determine a number of infinite families of double twist knots for which $\gamma_4$ is 2 or 3, respectively. Finally, in Section~\ref{twist knots} we focus on twist knots, again providing infinite families where we can determine $\gamma_4$. This extends the work of Feller and Golla~\cite{feller2021}.
   
The results in this paper do not finish the story for the nonorientable four genus of double twist knots. We highlight the edges of our knowledge and opportunities for future investigation via the following examples. (For more, see Tables~\ref{table1} and~\ref{table2}.)
\begin{itemize}
\item The knot $C(12, 2)$ is the lowest crossing number example  where we are not able to compute $\gamma_4$. The value is either 1 or 2.
\item The knot $C(12, -6)$ is the lowest crossing number example where we know that $\gamma_4$ is either 2 or 3, but we cannot determine the value. 
\item The knot $C(10, -6)$ is the lowest crossing number example  where we only know that $\gamma_4$ is 1, 2, or 3, but we cannot determine the value. 
\item We did not find any double twist knots $C(m,n)$ with $m>0$, $n<0$ and $\gamma_4=3$. Are there any?
\end{itemize}


\subsection*{Acknowledgements}
We thank the referee for useful comments leading to the reorganization and shortening of the paper.

\section{General methods for determining $\gamma_4$}\label{general methods section}

We use well-known methods for determining upper and lower bounds for $\gamma_4$ and describe these in this section. For a  knot $K$, an upper bound is found by exhibiting a specific nonorientable surface in $B^4$ with known first Betti number that spans $K$. One way to do this is to find such a surface in $S^3$ which can then be pushed into $B^4$. Thus the crosscap number $\gamma_3$ provides an upper bound for $\gamma_4$. Alternatively, we can also construct such a surface by finding a {\it band move} from $K$ to a slice knot. A lower bound for $\gamma_4$ is found using algebraic obstructions that ultimately depend on Donaldson's Diagonalization Theorem. 

To produce a nonorientable surface in $B^4$ spanning a knot $K$, we make use of the following construction. Given an oriented knot or link $K$ we may change $K$ to $K'$ by {\it attaching a band} $b=I \times I$ to $K$ as follows. First embed $b$ in $S^3$ so that $b \cap K=I\times \partial I$. Next, remove $I\times \partial I$ from $K$ and replace it with $\partial I \times I$ to create $K'$. Call this an {\it orientable} band move if some orientation of $b$ produces an orientation on $\partial b$ that agrees with the orientation on $K$ and  a {\it nonorientable} band move otherwise. If $K$ is a knot, an orientable band move will produce a link $K'$ of two components, while a nonorientable band move will produce a knot $K'$.  Note that if we can change $K$ to $K'$ by attaching a band, then we can likewise change  $K'$ to  $K$ by attaching a band, where the band is the same, but with its two interval factors reversed.

Suppose a band move on a knot $K$ produces the unlink of two components, the boundary of two disjointly embedded disks. Placing the band in general position in $S^3$ with respect to these two disks, we see that $K$ bounds an immersed disk with only ribbon singularities. Hence $K$ is slice. Similarly, if a nonorientable band move on a knot $K$ produces a slice knot, then $K$ bounds a M\"obius band in $B^4$. If two band moves take a knot $K$ to a slice knot, then $K$ bounds a nonorientable surface in $B^4$ with first Betti number 2.

Later, when we focus our attention on double-twist knots, we will consider band moves of a special kind that always produce 2-bridge knots. This will allow us to take advantage of Lisca's determination of exactly which 2-bridge knots are slice.

Several authors have formulated algebraic obstructions to a knot $K$ bounding a nonorientable surface in $B^4$ with small first Betti number. One of the most simply stated is the  following result which follows from Yasuhara \cite{Yasuhara}.

\begin{theorem}\label{Yasuhara lemma}
If $\sigma(K) +4\mathrm{Arf}(K)\equiv 4\Mod{8}$, then $\gamma_4(K) \geq 2$.
\end{theorem}

\begin{proof} In \cite{Yasuhara}, Yasuhara proves that if $\gamma_4(K) \le 1$, then 
$4\mathrm{Arf}(K)-\sigma(K)$ is $0$ or $\pm 2$ modulo $8$. If $K^*$ is the mirror image of $K$, then $\mathrm{Arf}(K^*)=\mathrm{Arf}(K)$ and $\sigma(K^*)=-\sigma(K)$. Hence, if $4\mathrm{Arf}(K)+\sigma(K) \equiv 4 \Mod{8}$, then $4\mathrm{Arf}(K^*)-\sigma(K^*) \equiv 4 \Mod{8}$. Therefore, $\gamma_4(K)=\gamma_4(K^*) \ge 2$ by Yasuhara's result.
\end{proof}

To determine $\gamma_4$ for all prime knots to nine crossings Jabuka and Kelly develop an obstruction that we will make heavy use of. We give a brief sketch here of how their obstruction is obtained.
Full details can be found  in \cite{JK}. 

Let $F$ be obtained from a checkerboard spanning surface $F'$ for some diagram $D$ of $K$ that has been pushed into $B^4$ and let $\Sigma$ be a hypothetical M\"obius band in $B^4$ that spans $K$. Let $W(F)$ and $W(\Sigma)$ each be 2-fold covers of $B^4$ branched along the surfaces $F$ and $\Sigma$, respectively. Both of these 4-manifolds have the same boundary $Y$,  which is the 2-fold cyclic cover of $S^3$ branched over $K$. Thus we may glue them together along their common boundary to obtain the closed 4-manifold
$$X=W(F)\cup _Y -W(\Sigma).$$
Under the right conditions, the intersection form on $X$ will be definite and hence diagonalizable by the work of Donaldson \cite{Donaldson2}. Gordon and Litherland \cite{Gordon-Litherland} show that the intersection form on $W(F)$ is given by the Goeritz matrix $G_\epsilon$ associated to $F'$. Work of Gilmer~\cite{G} and Gilmer--Livingston~\cite{GL} imply that the intersection form on $W(\Sigma)$ is given by the $1 \times 1$ matrix $[\pm\ell]$, where $\ell \in \mathbb N$, $\ell$ divides $\det K$, and $\det K/\ell $ is a square integer. This implies that the form $G_\epsilon \oplus [\epsilon \ell]$ will embed in the diagonal form $\epsilon \text{Id}$. Therefore, if no such embedding exists, then $K$ cannot bound a M\"obius band in $B^4$. The following lemma allows us to simplify the question of whether such an embedding exists to the case of $\ell=\det(K)$.

\begin{lemma}\label{need only check for ell equal det K}
Suppose $G$ is an $r \times r$, $\epsilon$-definite matrix, $\ell$ is a natural number, and $j$ is an integer. If there exists an embedding
$$\phi: (\mathbb Z^{r+1}, G \oplus[\epsilon \ell]) \to (\mathbb Z^{r+1}, \epsilon \mathrm{Id}),$$ 
then there exists an embedding
$$\phi': (\mathbb Z^{r+1}, G \oplus[\epsilon \ell j^2]) \to (\mathbb Z^{r+1}, \epsilon \mathrm{Id}).$$  
\end{lemma}
\begin{proof}
Let $\{e_i\}_{i=1}^{r+1}$ and $\{e_i'\}_{i=1}^{r+1}$  be bases for $\mathbb Z^{r+1}$    such that $e_i \cdot e_j$ is the $(i,j)$-entry of $G \oplus[\epsilon \ell]$  and $e_i' \cdot e_j'$ is the $(i,j)$-entry of $G \oplus[\epsilon \ell j^2]$  Given that $\phi$ exists, define $\phi'$ as
$$\phi'(e_i')=\left \{\begin{array}{ll} 
\phi(e_i),&\text{ if $i<r+1$,}\\
j\phi(e_{r+1}),&\text{ if $i=r+1$.}
\end{array} \right.
$$
We now have that, $e_i'\cdot e_j'=\phi'(e_i')\cdot \phi'(e_j')$ so that $\phi'$ preserves the linking form. Moreover, $G \oplus[\epsilon \ell j^2]$  is $\epsilon$-definite, so $\phi'$ is an embedding.
\end{proof}

Combining the results of Jabuka and Kelly \cite{JK} with Lemma~\ref{need only check for ell equal det K} gives the following theorem.

\begin{theorem}\label{JK theorems}
Suppose that $K$ is a knot with alternating diagram $D$ having positive and negative definite Goeritz matrices, $G_+$ and $G_{-}$, respectively.
\begin{enumerate}

\item If $\sigma(K)+4 \,\mathrm{Arf}(K)\equiv -2 \Mod{8}$ and there is no embedding 
\begin{equation*}\label{phi I plus embedding}
\phi:(\mathbb Z^{\mathrm{rank}(G_+)+1},G_+\oplus[\det(K)])\to (\mathbb Z^{\mathrm{rank}(G_+)+1}, \mathrm{Id}),
\end{equation*}
then $\gamma_4(K)\ge 2$.

\item If $\sigma(K)+4 \,\mathrm{Arf}(K)\equiv 2 \Mod{8}$ and there is no embedding 
\begin{equation*}\label{phi I minus embedding}
\phi:(\mathbb Z^{\mathrm{rank}(G_-)+1},G_-\oplus[-\det(K)])\to (\mathbb Z^{\mathrm{rank}(G_-)+1}, -\mathrm{Id}),
\end{equation*}
then $\gamma_4(K)\ge 2$.

\item If $\sigma(K)+4 \,\mathrm{Arf}(K)\equiv 0 \Mod{8}$ and there is no embedding 
$$\phi:(\mathbb Z^{\mathrm{rank}(G_\epsilon)+1},G_\epsilon\oplus[\epsilon \det(K)])\to (\mathbb Z^{\mathrm{rank}(G_\epsilon)+1}, \epsilon\mathrm{Id})$$ for both $\epsilon=1$ and $\epsilon=-1$, 
then $\gamma_4(K)\ge 2$.

\item If $\sigma(K)+4 \,\mathrm{Arf}(K)\equiv 4 \Mod{8}$ and there is no embedding
\begin{equation*}\label{phi II embedding}
\phi:(\mathbb Z^{\mathrm{rank}(G_\epsilon)},G_\epsilon)\to (\mathbb Z^{\mathrm{rank}(G_\epsilon)+2}, \epsilon\mathrm{Id})
\end{equation*}
for both $\epsilon=1$ and $\epsilon=-1$, then $\gamma_4(K)\ge 3$.
\end{enumerate}
\end{theorem}

Another important obstruction was introduced by Murakami and Yasuhara in \cite{Murakami_Yasuhara}. Restricting to double twist knots, their result states that if $C(m,n)$ bounds a locally flat M\"obius band in $B^4$, then $m$ or $-m$ is a quadratic residue modulo $k$ where $k$ is an integer such that $|mn+1|=kj^2$ and $\gcd(k,j)=1$. Thus, if $m$ and $-m$ are not quadratic residues for all such $k$, then $\gamma_4 \ge 2$. However, in the case of double twist knots, we did not make use of this obstruction for two reasons. The first is the difficulty in finding infinite families of double twist  knots that meet the quadratic residue hypothesis. Secondly, for double twist knots $C(m,n)$ with  $1\le m\le300$, $-300\le n \le 300$, and $n$ even, our calculations produced no cases of Murakami--Yasuhara producing a better lower bound than the already mentioned obstructions of Yasuhara and Jabuka-Kelly.  An interesting example is the double twist knot $K=C(5,14)$. By Theorem~\ref{double twist knots with gamma4 equal to 2} (8), we have $\gamma_4(K)=2$. But, because $17^2 \equiv 5 \Mod{71}$, $5$ is a quadratic residue modulo $71$ and we see that Murakami and Yasuhara's result fails to obstruct the existence of a locally-flat embedded M\"obius band in $B^4$ spanning $C(5,14)$.   So  while $K$ cannot bound a smoothly embedded M\"obius band in $B^4$, it remains an  interesting question  if $C(5,14)$ bounds a locally flat M\"obius band.


\section{Properties of Double Twist Knots}\label{Properties of Double Twist Knots}

In this section, we list various invariants of double twist knots that we will need in order to apply the general methods described in Section~\ref{general methods section}. We first consider the crosscap number of $C(m,n)$. The following proposition can be derived directly from \cite{Hirasawa_Teragaito:2006}, where an algorithm for determining the crosscap number of any 2-bridge knot is given.

\begin{proposition}[Hirasawa-Teragaito]\label{crosscap number of double-twist knots} If  $m>1$ and $n$ is even and nonzero, then
$$
\gamma_3(C(m,n))=\left \{ \begin{array}{ll}
2,&\text{if 
$m$ is odd or $m=2$ or $n=\pm 2$,}\\
3,& \text{if $m$ is even,  $m \neq 2$, and  $|n| \neq 2$.}
\end{array} \right.
$$
\end{proposition}

In order to apply the Jabuka and Kelly obstructions from Theorem~\ref{JK theorems} we need to determine the signature and Arf invariants of $C(m,n)$. Both are well known and appear in several places in the literature.

\begin{proposition}\label{sigma}
Assuming $m>0$ and $n$ is even, then the values of $(\sigma+4\mathrm{Arf})(C(m,n)) \Mod{8}$ are given as follows.
\begin{center}
\begin{tabular}{|r|r|r||r|r|} \hline
 & \multicolumn{2}{|c||}{$n\ge 0$} & \multicolumn{2}{|c|}{$n<0$} \\ \hline
 {\rm (mod 4)} & $n \equiv 0$  & $n \equiv 2$ & $n \equiv 0$ & $n \equiv 2$ \\ \hline
 $m \equiv 0$ & $0$ & $0$ & $2$ & $2$ \\
 \hline
 $m \equiv 1$  & $0$ & $-2$ & $2$ & $0$ \\
 \hline
 $m \equiv 2$  & $0$ & $4$ & $2$ & $-2$ \\
\hline
$m \equiv 3$  & $0$ & $2$ & $2$ & $4$ \\
\hline
 \end{tabular}
\end{center}
\end{proposition}
\begin{proof}

If $m>0$ and $n$ is even, then the $\mathrm{Arf}$ invariant of $C(m,n)$ is as follows (see \cite{Madeti}):

\begin{align*}
    \mathrm{Arf}(C(m,n))&\equiv \left \{ \begin{array}{lll} mn/4 & \Mod{2},&\text{ if $m$ is even,}\\
    n(2m+n)/8& \Mod{2},&\text{ if $m$ is odd.}
    \end{array}\right.
\end{align*}

\noindent The signature of $C(m,n)$ is easily computed from certain continued fraction expansions of  $(mn+1)/n$ as described in Chapter 12 of Burde and Zeischang \cite{BZ:2003} or the work of Gallaspy and Jabuka \cite{GJ:2015}. 

$$\sigma(C(m,n))=\left \{ \begin{array}{rl} 
n, & \text{ if $m$ is odd and $n>0$,}\\
0, &\text{ if $m$ is even and $n>0$,}\\
n+2, & \text{ if $m$ is odd and $n<0$,}\\
2, &\text{ if $m$ is even and $n<0$.}
\end{array} \right.
$$

\smallskip
 
As already mentioned, we regard a knot $K$ and its mirror image $K^*$ as equivalent.   It is well know that $\sigma(K^*)=-\sigma(K)$ while $\mathrm{Arf}(K^*)=\mathrm{Arf}(K)$. Thus the  quantity $\sigma +4\mathrm{Arf}$ can change when $K$ is replaced with its mirror image and is therefore not  an invariant of knot type. However, because $4$ and $-4$ are congruent modulo 8, it follows that  $\sigma +4\mathrm{Arf}$ is negated when taking mirror image and is thus unchanged, modulo 8,  when equal to 0 or 4. Using the computations of the signature and Arf invariant given above, we  obtain the values of  $(\sigma +4\mathrm{Arf})(C(m,n)) \Mod{8}$ in  the table.
\end{proof}

The additional data needed to apply the Jabuka and Kelly obstructions of Theorem~\ref{JK theorems} are the positive and negative Goeritz matrices of $C(m,n)$. With $m>1$ and $n$ even and nonzero, the diagrams $C(m,n)$ and $C(m-1,1,-n-1)$ are alternating when $n>0$ and $n<0$, respectively. Using the checkerboard surfaces determined by these  alternating diagrams, it is straightforward to derive the associated Goeritz matrices as described in  \cite{Gordon-Litherland} (see also \cite{GJ:2015}).  To describe these matrices, it is helpful to define the following non-singular $k \times k$ matrix $T_k$.

$$T_k=  \left (
\begin{array}{rrrrr}
2 &-1&\cdots&& 0\\
-1&2\\
\vdots&&\ddots&&\vdots\\
&&&2&-1\\
0 &&\cdots &-1&2
\end{array}\right )_{k \times k}$$

\noindent
Here the diagonal contains all $2$s, the super and sub diagonals are all $-1$s, and all other entries are zero.

\begin{proposition} 
\label{Goeritz matrices}
Suppose $m>1$,  $n > 0$, and $n$ is even. Then  positive- and negative-definite Goeritz matrices $G_+$ and $G_-$, respectively, for $C(m,n)$ are
 $$
\begin{array}{cc}

G_+=  \left (
\begin{array}{rcr|c}
& & & 0 \\
& T_{n-1} & & \vdots \\ 
& & & -1  \\ \hline
0 & \cdots & -1  & m+1
\end{array}\right )
&

G_-=  \left (
\begin{array}{rcr|c}
& & & 0 \\
& -T_{m-1} & & \vdots \\ 
& & & 1  \\ \hline
0 & \cdots & 1  & -(n+1) 
\end{array}\right )
\end{array}
$$

\noindent and   positive- and negative-definite Goeritz matrices $\Gamma_+$ and $\Gamma_-$, respectively, for $C(m,-n)$ are
$$
\begin{array}{cc}

\Gamma_+=\left ( \begin{array}{ccc}
m&~&-1\\
-1&~&n
\end{array} \right )
&

\Gamma_-=  \left (
\begin{array}{rcr|c|rcr}
& & & 0 & & &\\
& -T_{m-2} & & \vdots & & 0 & \\ 
& & & 1  \\ \hline
0 & \cdots & 1  & -3 &1&\cdots&0\\ \hline
&&&1 \\
& 0 &&\vdots&&-T_{n-2}\\
&&&0
\end{array}\right )
\end{array}
$$
\end{proposition}

These matrices can also be represented by weighted graphs. For example, $G_-$ with $n>0$ would appear as below. 

\bigskip

\centerline{\includegraphics[width=2in]{Figures/graph2.pdf}}

\section{Double twist knots with $\gamma_4=0 $ or $\gamma_4=1$}\label{gamma_4 is 0 or 1}

In this section, we first determine exactly which double twist knots are slice ($\gamma_4=0$), using the work of Lisca \cite{Lisca}. We then use this result to find families of double twist knots with $\gamma_4 =1$.

\begin{figure}
\includegraphics[width=11cm]{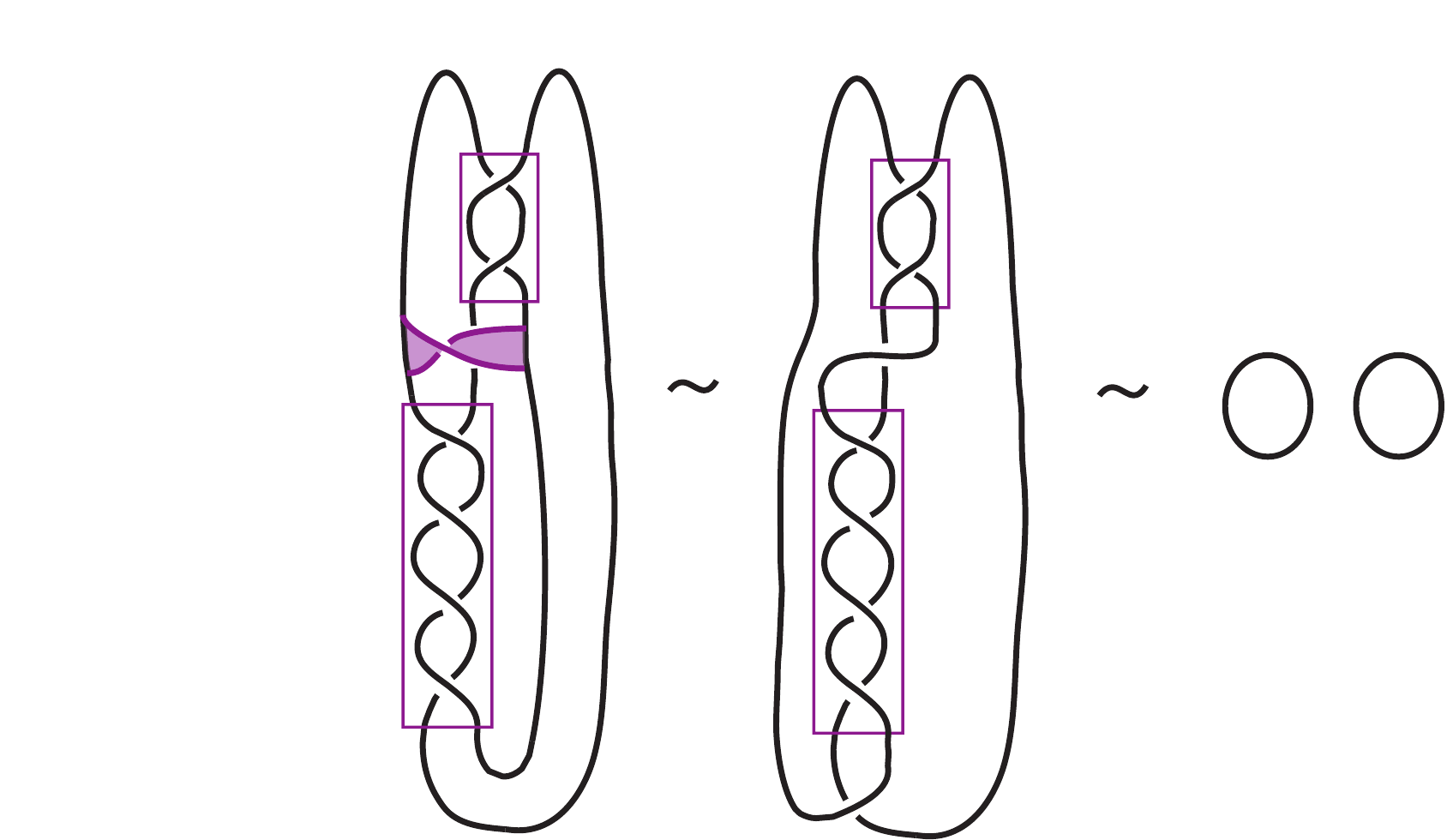}

\caption{A band move on $C(2,4)$ which results in the unlink.  A similar move applied to $C(m,m+2)$, for any $m$, gives the same result.}
\label{band move to unlink}
\end{figure}

\begin{theorem}\label{double twist knots with gamma4 equal to 0} If  $m>1$ and $n$ is even and nonzero, then $\gamma_4(C(m,n))=0$  if and only if $|m-n|=2$ or $(m,n)=(5,-2)$. 
\end{theorem}

\begin{proof}
In Figure~\ref{band move to unlink}, we exhibit an orientable band move on $C(m,m+2)$ which results in the unlink, showing that these knots are slice.  This band move is found in Figure 2 of~\cite{Lisca}. The knot $C(5,-2)$ is therefore slice because $C(5,-2)=C(2,4)$. 

We use Lisca's classification of slice 2-bridge knots \cite{Lisca} to complete the proof. Lisca works with the {\it subtractive} continued fraction $[b_1,b_2,\dots, b_k]^-$  given by
$$[b_1,b_2,\dots, b_k]^- = b_1- \frac{1}{b_2 - \displaystyle \frac{1}{\ddots - \displaystyle\frac{1}{b_k}}}$$ with the additional assumption that $b_i\ge 2$ for all $i$.   He also denotes a sequence of $j$ consecutive $b_i$s all equal to 2 as $2^{[j]}$. In a series of lemmas, Lisca proves that a 2-bridge knot $K_{p/q}$ is slice if and only if the subtractive continued fraction of either $p/q$, $p/q'$, $p/(p-q)$, or $p/(p-q')$ (where $q'$ is the multiplicative inverse of $q$ modulo $p$) appears in one of seven families (see Lemmas~8.1-5 in \cite{Lisca}). Lecuona observed that Lisca missed one additional family that can be found in Remark 3.2(II)(3) of \cite{Lecuona} and should have appeared as an additional case in Lisca's Lemma~8.4. Based on this work, we establish when a double twist knot is slice by finding the four subtractive continued fractions for $C(m,n)$ and then determine if any one of the continued fractions appear in one of the eight slice 2-bridge knot families for some choice of $m$ and $n$.

Suppose that $K=C(m,n)$ is a nontrivial, double twist knot with $m>0$ and $n$ even and nonzero. If $m=1$, then $K$ is a nontrivial, torus knot and cannot be slice because $\sigma(K) \ne 0$. Therefore, we assume $m \ge 2$. If $n \ge 2$, then the subtractive continued fractions of $C(m,n)$ are $[m+1,2^{[n-1]}]^-$, $[2^{[n-1]},m+1]^-$, $[2^{[m-1]},n+1]^-$, and $[n+1,2^{[m-1]}]^-$. The only family of slice 2-bridge knots which could possibly contain these continued fractions is $[c_1+1,2^{[c_1+1]}]^-$ (Lemma~8.1 \cite{Lisca}).
We now consider the necessary conditions on $m$ and $n$ for this to occur. If $[m+1,2^{[n-1]}]^-=[c_1+1,2^{[c_1+1]}]^-$, then $c_1=m$ and $n=m+2$ which gives the desired result. Similarly, if $[n+1,2^{[m-1]}]^-=[c_1+1,2^{[c_1+1]}]^-$, then $c_1=n$ and $m=n+2$ which again gives the desired result. 

Next we consider $C(m,-n)$ with $n \ge 2$. For $C(m,-n)$, we have subtractive continued fractions $[m,n]^-$, $[n,m]^-$, $[2^{[m-2]},3,2^{[n-2]}]^-$, and $[2^{[n-2]},3,2^{[m-2]}]^-$. Neither $[m,n]^-$ nor $[n,m]^-$ are in one of the slice families of 2-bridge knots and the latter two continued fractions can only be in the family $[c_1+1,2^{[c_1+1]}]^-$ when $m=2$ or $n=2$, respectively. In the first case, if $[3,2^{[n-2]}]^-=[c_1+1,2^{[c_1+1]}]^-$, then $n=5$ but this contradicts our assumption that $n$ is even. In the final case, if $[3,2^{[m-2]}]^-=[c_1+1,2^{[c_1+1]}]^-$ then $m=5$. This gives the stevedore's knot $C(5,-2)$ which is equivalent to $C(2,4)$.
\end{proof}

\begin{figure}
\centering
  \parbox[c]{1.75in}{\includegraphics[width=1.752in]{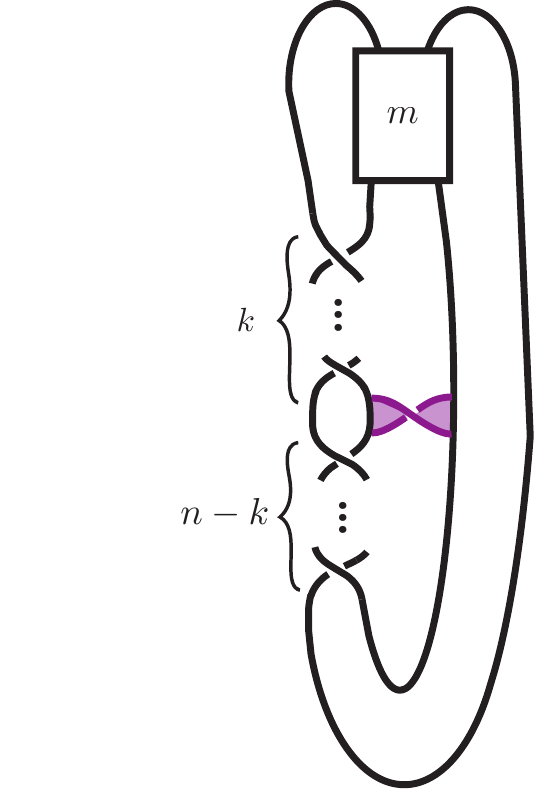}}
  \hspace{1in}
  \parbox[c]{2in}{\includegraphics[width=2in]{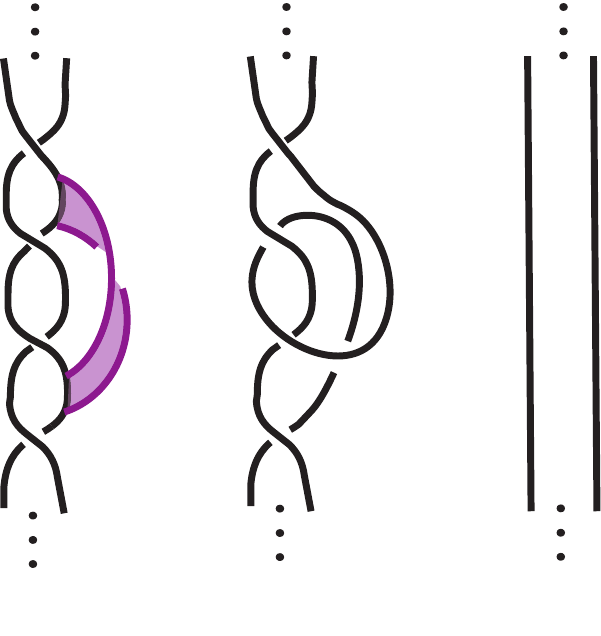}}
  \captionof{figure}{Two types of band moves. On the left, a horizontal band move with $n>0$ and $\epsilon=-1$.  On the right, a band move that removes four half-twists.}
  \label{band moves}
\end{figure}

In an effort to find band moves from double twist knots to slice knots,   we systematically studied ``horizontal'' band moves  of the form 
$C(m,n) \rightarrow C(m,k,\epsilon,n-k)$,
where $k$ is any integer and $\epsilon=\pm 1$. Such a band move is shown on the left in Figure~\ref{band moves}. Depending on the values of $m, n$ and $k$, the move may or may not be nonorientable.  (Notice that the band move in Figure~\ref{band move to unlink} is equivalent to a horizontal band move with $k=1$ and $\epsilon=-1$.) The advantage of limiting ourselves to band moves of this kind is that they produce 2-bridge links, and using Lisca's classification, we can tell precisely which band moves yield slice knots. We also make use of the band move shown on the right in Figure~\ref{band moves} which was introduced by Kearney~\cite{Kearney}. This move can be used to change either $m$ or $n$ in $C(m,n)$ by $\pm 4$.
\bigskip

We are now able to identify several families of double twist knots for which $\gamma_4=1$. Presumably other types of band moves remain to be discovered that could identify many more such families.

\begin{theorem}\label{double twist knots with gamma4 equal to 1}
If $m>0$, $n$ is even and nonzero,  and $(m,n)$ is one of the following, then $\gamma_4(C(m,n))=1$.  

\begin{enumerate}
\item $(1, n)$, except the unknot $(1,-2)$, 
\item  $(2,-6)$, $(2,-10)$, $(3,8)$, $(5,-6)$, $(6,-6)$, 
\item $(4,n)$, except $(4,2)$ and $(4,6)$,
\item $(m,\pm 4)$, except $(6,4)$, or
\item $(m,n)$ where $|m-n|\in \{1, 3, 6, 7\}$, except $(5,-2)$.
\end{enumerate}
\end{theorem}

\begin{proof}
None of the above knots are slice by Theorem~\ref{double twist knots with gamma4 equal to 0}.  The knots $C(1,n)$ are $(2,1-n)$-torus knots which bound M\"obius bands in $S^3$. For the remaining cases,  it suffices to produce a band move  that takes $K=C(m,n)$ to a slice knot, thus showing that $K$ bounds a M\"obius band in $B^4$. Note first that $(2,-6)$ and $(2,-10)$ are equivalent to $(5,2)$ and  $(9,2)$, respectively, by the clasp move and hence are covered by the final case.
The case of the knot $C(5,-6)$ is already known (see~\cite{Ghanbarian2020}), but we include it here for completeness. A single band move applied to $C(5,-6)$ produces the stevedore's knot $6_1=C(4,2)$, which is slice:
$$C(5,-6) \rightarrow C(5,-4,1,-2)=K_{9/2}=C(4,2)$$
The following band move takes $C(6,-6)$ to a slice knot:
$$C(6,-6) \rightarrow C(6,-2,1,-4)=K_{9/2}$$
A single band move applied to $C(3,8)$ produces the  knot $K_{25/7}=8_9$:
$$C(3,8)\rightarrow C(3,3,-1,5)=K_{25/7} = 8_9$$
The knot $K_{25/7}$ is slice because $25/18=[2^{[2]},3,4]^-$ which appears in a slice 2-bridge knot family from Lemma~8.4 of  \cite{Lisca}.

The knots $C(4,n)$ and $C(m,\pm4)$ can be taken to the unknot with a single band move of the kind shown on the right in Figure~\ref{band moves}. For the second family, this was first found by Kearney~\cite{Kearney}.

The following band moves together with the equivalence $C(m,n)=C(n,m)$ of double twist knots demonstrate the final cases $|m-n|\in\{1,3,6,7\}$.
\begin{itemize}
\item $C(m+1,m) \rightarrow C(m+1,0,1,m)=C(m+2,m)$
\item $C(m+3,m) \rightarrow C(m+3,0,-1,m)=C(m+2,m)$
\item  $C(m,m+6)\rightarrow C(m,m+4,-1,2)=C(m,m+2)$
\item $C(m,m+7)\rightarrow C(m,5,-1,m+2)$
\end{itemize}

\noindent The knot $C(m+2,m)$ is slice by Theorem~\ref{double twist knots with gamma4 equal to 0}. The knot $C(m,5,-1,m+2)$ is slice because it has a subtractive continued fraction $[m+1, 2^{[2]},3,2^{[m-1]}]^-$ which appears in a slice 2-bridge knot family in Lemma~8.1 of \cite{Lisca}. Alternatively, the band move
$C(m,5,-1,m+2)\rightarrow C(m,2,-1,3,-1,m+2)=K_0$ results in the unlink, showing that the knot is slice.
\end{proof}

It is worth noting that given any knot $K$ which admits a band move to a slice knot, there may be more than one such band move. For example, the following band move takes $C(m,\pm 4)$ to the unknot:
 $$ C(m, \pm 4)\rightarrow  C(m,\pm 2,\mp 1,\pm 2)=K_{1/0},$$
 as does Kearney's move shown in Figure~\ref{band moves}. Kearney's move offers an alternate proof that $C(m,n)$ with $|m-n|=6$ bounds a M\"obius band because the distance between $m$ and $n$ can be changed from 6 to 2.


\section{Double twist knots with $\gamma_4 = 2$}\label{gamma_4 is 2}

In this section, we determine many infinite families of double twist knots with $\gamma_4= 2$. We begin by considering the obstructions of Jabuka and Kelly.

\begin{lemma}\label{lemma1}
Let $G_\epsilon$ be an $\epsilon$-definite matrix represented by the weighted graph below where $k>0$ and $\ell\geq 0$ and even. Denote the vertices of the graph by $\{e_i\}_{i=1}^{k+\ell+1}$. If there is an embedding
$$\phi:(\mathbb Z^{k+\ell+1},G_\epsilon) \rightarrow (\mathbb Z^r, \epsilon \mbox{\rm Id}),$$
then, up to a change of basis, we may assume that $\phi(e_i)=f_i - f_{i+1}$ for all $i \neq k+1$, where $\{f_i\}_{i=1}^{r}$ denote a basis for the codomain with $f_i \cdot f_j = \epsilon \delta_{ij}$. In particular, this implies that the codomain has dimension  $r\geq k +1$ if $\ell = 0$ and dimension $r\geq k + \ell + 2$ if $\ell>0$.

\begin{center}
   \includegraphics[width=6.5cm]{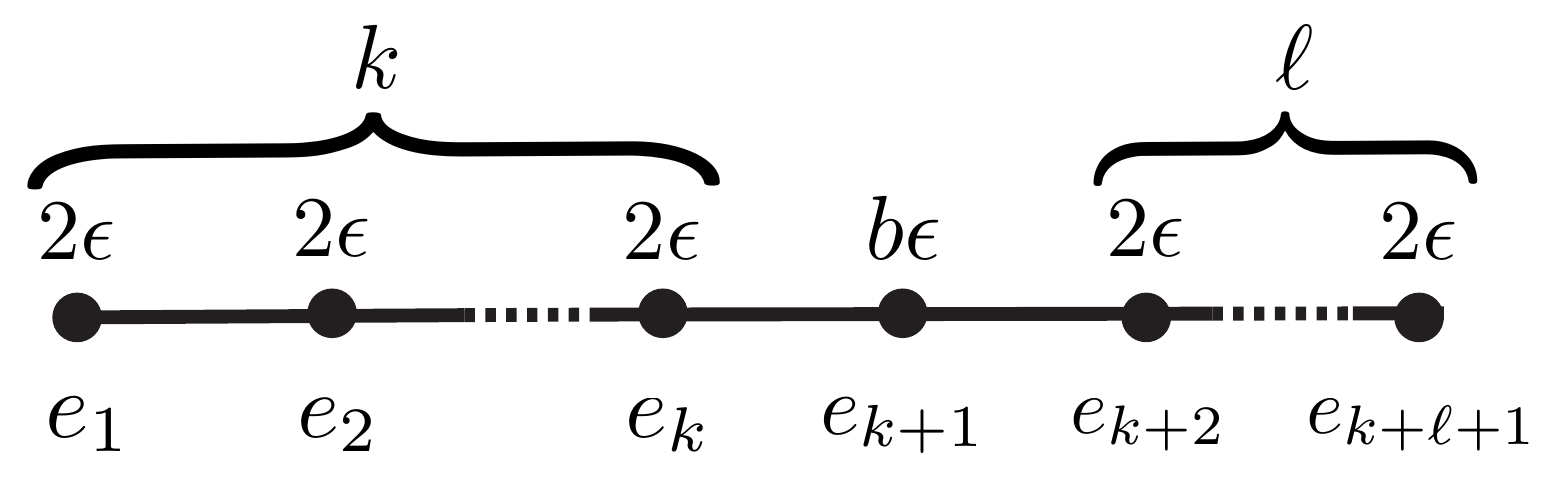}
\end{center} 
\end{lemma}

\begin{proof}
We provide the proof for a positive definite matrix and leave the remaining case to the reader. The fact that $e_m \cdot e_m=2$ for all $m\neq k+1$ implies that $\phi(e_m)$ must be expressed as $\phi(e_m)=a_if_i+a_jf_j$ for some $i<j$ with  $|a_i|=|a_j|=1$ for all $m\neq k+1$. We wish to show that we can acheive $\phi(e_m) = f_m - f_{m+1}$ for $m \neq k+1$ by changing the basis elements of the codomain if necessary.

First, we consider $\phi(e_1)=a_if_i+a_jf_j$. By reordering and negating the basis elements of the codomain if necessary (which by abuse of notation, we will still continue to call $\{f_i\}_{i=1}^{r}$), we can assume $\phi(e_1)=f_1-f_2$, as desired.

Next consider $\phi(e_2)$ (assuming that $k\geq 2$), where $\phi(e_2)=a_if_i+a_jf_j$ for some $i<j$ and  $|a_i|=|a_j|=1$. Because $e_1\cdot e_2=-1$, it follows that $i \in \{1,2\}$ and $j >2$.   If $f_i = f_2$, then we have $\phi(e_2)=f_2+a_jf_j$ for some $j>2.$ On the other hand, if $f_i = f_1$, then we rename $f_1$ as $-f_2$ and $f_2$ as $-f_1$. We still have $\phi(e_1)=f_1-f_2$ but now $\phi(e_2)=f_2+a_jf_j$ for some $j>2$.  Now, if necessary, we may again change the basis  so that $f_j=\pm f_3$ and arrive at $\phi(e_2)=f_2-f_3$, as desired. 

Assuming $k\geq 3$, we next consider $\phi(e_3)$. Because $e_1\cdot e_3 = 0$ and $e_2\cdot e_3=-1$, it follows that either $\phi(e_3)=-f_1-f_2$ or $\phi(e_3) = f_3+a_jf_j$ with $j>3$. However, in the first case the embedding cannot be extended to $e_4$. To see this, suppose $\phi(e_4)=\Sigma_{i=1}^{r}a_if_i$.  Because $e_1\cdot e_4=0$ and $e_3 \cdot e_4=-1$, it follows that $a_1-a_2=0$ and $-a_1-a_2=-1$ which together have no integer solutions. Thus we must have  $\phi(e_3) = f_3+a_jf_j$ with $j>3$. Changing the basis of the codomain, if necessary, we may assume $\phi(e_i)=f_i-f_{i+1}$ for $1\le i \le 3$.

If $k>3$, we continue by induction. Assume that we can change the basis of the codomain, if necessary, so that $\phi(e_i)=f_i-f_{i+1}$ for $1 \le i\le m<k$ and that $m\ge 3$. Suppose that $\phi(e_{m+1})=\Sigma_{i=1}^{r}a_if_i$. Because $e_i \cdot e_{m+1}=0$ for $i< m$ and $e_m\cdot e_{m+1}=-1$, it follows that $a_1=a_2=\cdots=a_m=a_{m+1}-1$. We now have  $2=e_{m+1}\cdot e_{m+1}=\phi(e_{m+1})\cdot \phi(e_{m+1})\ge ma^2+(a+1)^2$ where $a=a_1=\cdots=a_m$. Because $m\ge 3$, we must have $a=0$ and hence $\phi(e_{m+1})=f_{m+1}+a_j f_j$ for some $j>m+1$. By changing the basis again, if necessary, we may assume $\phi(e_i)=f_i-f_{i+1}$ for $1 \le i\le m+1$, thus completing the inductive step.

If $\ell=0$, then the lemma is complete. If not, then $\ell \ge 2$. Consider any $e_i$ with $i>k+1$. Suppose $\phi(e_i)=\Sigma_{j=1}^{r}a_jf_j$. The dot product of this vector with $f_j-f_{j+1}$ is zero for all $1\le j\le k$. Hence $a_1=a_2=\cdots=a_{k+1}$. We now have
 $2=\phi(e_i)\cdot \phi(e_i)\ge (k+1)a^2,$
 where $a=a_1=\cdots=a_{k+1}$.
Hence, if $k=1$, then $a=0$ or $a=\pm1$ and so either $\phi(e_i)=\pm(f_1+f_2)$ or $\phi(e_i)=\Sigma_{j=k+2}^{r}a_jf_j$ and if $k>1$, then $a=0$ and $\phi(e_i)=\Sigma_{j=k+2}^{r}a_jf_j$.

If $k=1$, then none of the $e_i$ with $i>k+1$ can have image $\pm(f_1+f_2)$. For, suppose there were an index $i$ with $i>k+1$  and $\phi(e_i) =\pm(f_1+f_2)$. Let $j=i\pm1$ such that  $k+1<j\leq k+\ell+1$. Such a $j$ exists because $\ell \ge2$. Now  the dot product of $\phi(e_i)$ and $\phi(e_{j})$  would  either be $\pm 2$ (if $\phi(e_{j})$ also had image $\pm(f_1+f_2)$) or zero (if $\phi(e_{j})$ had image $\Sigma_{j=k+2}^{r}a_jf_j$). But this contradicts the fact that $\phi(e_i)\cdot \phi(e_{j})=-1$.

So, regardless of the value of $k$,  if $i>k+1$, then $\phi(e_i)=\Sigma_{j=k+2}^{r}a_jf_j$ where the coefficients $a_j$ depend on $i$. Thus we may change the basis of the codomain, if necessary, so that we still have $\phi(e_i)=f_i-f_{i+1}$ for $1\le i\le k$ and now also have $\phi(e_{k+2})=f_{k+2}-f_{k+3}$. We now proceed by induction similar to before. Note that to start the induction when $\ell>2$, we must prove the first three cases of the inductive hypothesis and the proof of the third case depends on the existence of $e_{k+5}$, which we know exists because $\ell$ is even.
\end{proof}

We now proceed to use Lemma~\ref{lemma1} to establish the following results regarding embedding the Goeritz matrices of $C(m,n)$.

\begin{proposition}\label{m,n>0 phi embedding conditions}
Let $m>1$, $n>0$ and even, and  $G_\epsilon$  be the $\epsilon$-definite Goeritz matrix of $C(m,n)$ as given in Proposition~\ref{Goeritz matrices}. There is an embedding $\phi$ as given in Theorem~\ref{JK theorems} \eqref{phi I plus embedding}   if and only if there exist integers $x$ and $y$ such that
\begin{equation*}\label{xy plus condition}m=n x^2+2x+y^2.\end{equation*}
There is an embedding $\phi$ as given in Theorem~\ref{JK theorems} \eqref{phi I minus embedding} if and only if there exist integers $x$ and $y$ such that
\begin{equation*}\label{xy minus condition}n=m x^2+2x+y^2.\end{equation*}
\end{proposition}

\begin{proof}  We give the proof in the case of $G_+$. The reader can easily make the necessary changes for the case $G_{-}$.

Suppose $\phi$ exists. Let $\{e_i\}_{i=1}^{n+1}$  be a basis for the domain and $\{f_i\}_{i=1}^{n+1}$ be a basis for the codomain such that the pairing on $\{e_i\}_{i=1}^{n+1}$ is given by $G_+\oplus [mn+1]$, which we denote graphically by:\\

\begin{center}
   \includegraphics[width=6.5cm]{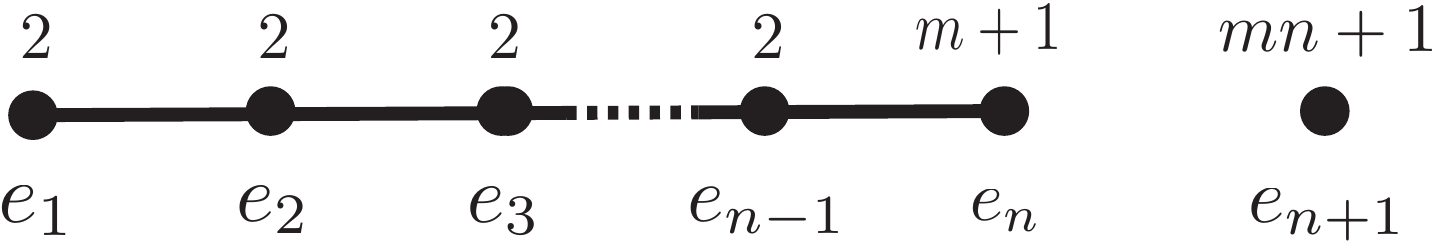}
\end{center}
Moreover, the pairing on $\{f_i\}_{i=1}^{n+1}$ is given by  $f_i\cdot f_j=\delta_{ij}$.\\

By Lemma~\ref{lemma1}, we may assume that  $ \phi(e_i)=f_i-f_{i+1}\text{ for } 1\le i \le n-1.$  We now determine $\phi(e_n)$. Suppose that $\phi(e_n)=\sum_{i=1}^{n+1} a_i f_i$. Because $e_i \cdot e_n=a_i-a_{i+1} = 0$ for $1\le i\le n-2$, it follows that
$$a_1=a_2=\cdots=a_{n-1}$$ and we rename $a_1=x$.  
Moreover because $e_{n-1} \cdot e_n=x-a_{n}= -1$, it follows that $a_{n} = x+1.$ And because $m+1=e_{n} \cdot e_n= (n-1)x^2 + (x+1)^2 +a_{n+1}^2$, we now have that $m+1=(n-1)x^2+(x+1)^2+a_{n+1}^2$, or equivalently, $m=nx^2+2x+y^2$ as desired, where we have renamed $a_{n+1}$ as $y$. 

 We now consider the converse. Suppose there exist integers $x,y$ such that $m=nx^2+2x+y^2$. Then we define:
 \begin{align*}
      \phi(e_i) &= f_i - f_{i+1} \text{ for } 1\le i\le n-1\\
      \phi(e_n) &=xf_1 + xf_2 + \cdots + xf_{n-1}+ (x+1)f_n + yf_{n+1}\\
      \phi(e_{n+1}) &= yf_1 + yf_2 + \cdots + yf_{n-1} + yf_n -(nx+1)f_{n+1}.
 \end{align*}
It remains to check that $e_{n+1}\cdot e_{n+1}=\phi(e_{n+1})\cdot \phi(e_{n+1})$.
Note that 
\begin{align*}
    \det K&=mn+1\\
    &=n^2x^2+2nx+ny^2+1\\
    &=ny^2+(nx+1)^2\\
    &=\phi(e_{n+1})\cdot \phi(e_{n+1}).
\end{align*}Thus $\phi$ respects the linking form. Because $G_+ \oplus[\det K]$ is positive definite, $\phi$ is an embedding.
\end{proof}

Given positive integers $m$ and $n$, it is easy to check if there exist integers $x$ and $y$ with $m=n x^2+2x+y^2$. One need only check for each integer $y$ with $0\le y\le \sqrt m$ if the equation  has integer solutions in $x$. 

Now we consider double twist knots $C(m,-n)$ with $m,n>1$ and $n$ even. Recall that the determinant of $C(m,-n)$ is $mn-1$. 

\begin{proposition}\label{m,-n>0 phi embedding conditions}
Let $m>1$, $n>0$ and even, and  $\Gamma_+$  be the positive-definite Goeritz matrix of $C(m,-n)$ as given in Proposition~\ref{Goeritz matrices}. Then there is an embedding $\phi$ as given in Theorem~\ref{JK theorems} \eqref{phi I plus embedding} if and only if there is an integral solution to the following system of equations
\begin{align}
\label{short XXT equations}
m &= x_{11}^2+x_{12}^2+x_{13}^2 \nonumber\\ 
n &=  x_{21}^2+x_{22}^2+x_{23}^2 \\\nonumber
-1& =x_{11} x_{21}+ x_{12} x_{22}+x_{13} x_{23}.
\end{align}
\end{proposition}
\begin{proof}
Let $\{e_1, e_2, e_3\}$  be a basis for $\mathbb{Z}^{3}$ with pairing given by $\Gamma_+\oplus [mn-1]$ as depicted below and $\{f_1, f_2, f_3\}$ be the standard basis for $\mathbb{Z}^{3}$ with $f_i\cdot f_j=\delta_{ij}$. 
\begin{center}
   \includegraphics[width=2.5cm]{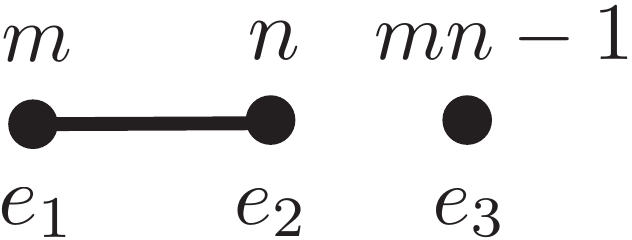}
\end{center}
If the embedding $\phi$ exists, then $\phi(e_i)=\Sigma_{j=1}^3 x_{ij}f_j$ for some integers  $x_{ij}$ such that
\begin{align*}
m &= e_1 \cdot e_1 = \phi(e_1) \cdot \phi(e_1) =  x_{11}^2+x_{12}^2+x_{13}^2 \\ 
n &= e_2 \cdot e_2 = \phi(e_2) \cdot \phi(e_2) =  x_{21}^2+x_{22}^2+x_{23}^2 \\
-1& = e_1 \cdot e_2 = \phi(e_1) \cdot \phi(e_2) = x_{11} x_{21}+ x_{12} x_{22}+x_{13} x_{23}
\end{align*}
This gives an integral solution to Equations~\ref{short XXT equations}. Conversely, if there is an integral solution to Equations~\ref{short XXT equations}, then define $\phi(e_1) =\Sigma_{j=1}^3 x_{1j}f_j$, $\phi(e_2) =\Sigma_{j=1}^3 x_{2j}f_j$, and $\phi(e_3) = \phi(e_1) \times \phi(e_2)$ (which is an integer valued vector). Then $\phi(e_1) \cdot \phi(e_1) =m$, $\phi(e_2) \cdot \phi(e_2) =n$, and $\phi(e_1) \cdot \phi(e_2) =-1$. Clearly,  $\phi(e_1) \cdot \phi(e_3) = \phi(e_2) \cdot \phi(e_3) =0$ by properties of the cross product. Moreover, if $\theta$ is the angle between $\phi(e_1)$ and $\phi(e_2)$, then 
$1=(\phi(e_1) \cdot \phi(e_2))^2 = m n \cos^2 \theta$. Thus,
$$\phi(e_3) \cdot \phi(e_3) = \left| \phi(e_1) \times \phi(e_2) \right|^2 = m n \sin^2 \theta = mn \left(1- \frac{1}{mn} \right) = mn-1.$$
Therefore, $\phi$ is an embedding because it preserves the pairing given by $\Gamma_+\oplus [mn-1]$ and $\Gamma_+\oplus [mn-1]$ is positive definite. 
\end{proof}

The following technical lemma is necessary for the final obstruction theorem.

\begin{lemma}\label{lemma2}
Let $G$ be a negative definite  matrix represented by the weighted graph below where $k>0$ and $\ell\geq 0$ is even. There is an embedding
$$\phi:(\mathbb Z^{k+\ell+1},G) \rightarrow (\mathbb Z^{k+\ell+2}, -\mbox{\rm Id}),$$ if and only if $(k,\ell)=(3,0)$, $k=2$, or $\ell = 2$.
\begin{center}
   \includegraphics[width=6.5cm]{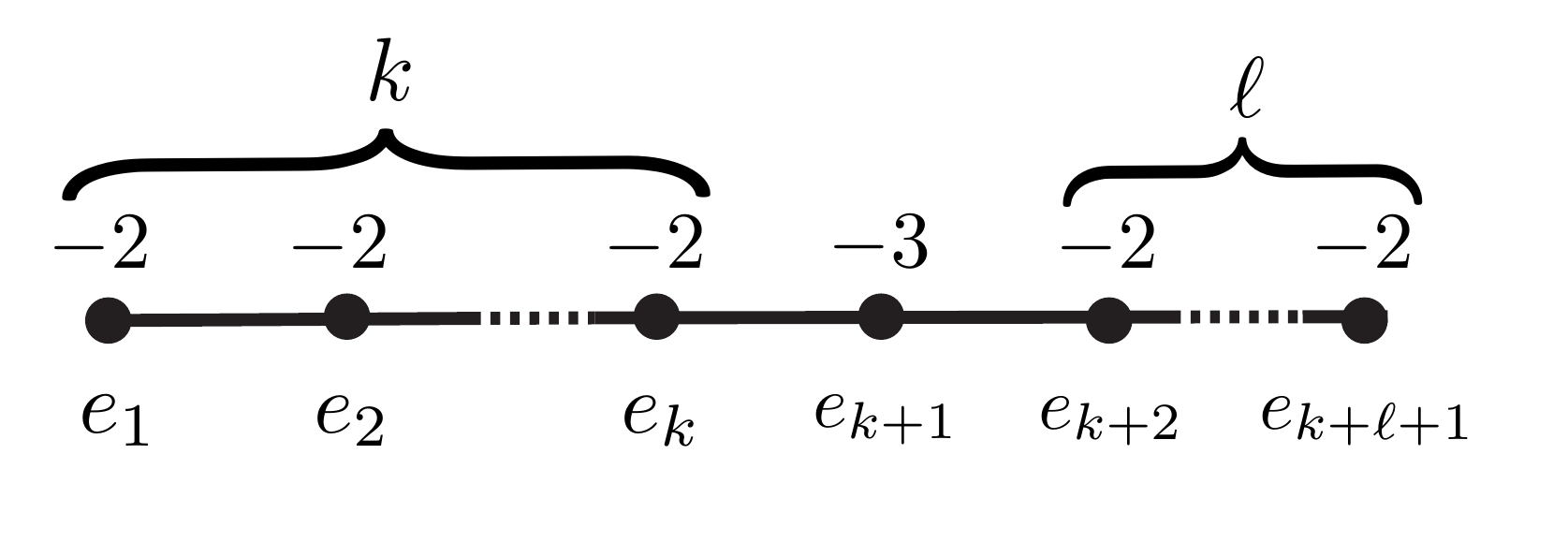}
\end{center}
\end{lemma}

\begin{proof} We first show the existence of embeddings when $(k,\ell)=(3,0)$, $k=2$, or $\ell = 2$. If $(k,\ell)=(3,0)$, then the following defines an embedding $\phi: \mathbb{Z}^4 \rightarrow \mathbb{Z}^5$.
$$\phi(e_i)=\left\{\begin{array}{ll}
f_i- f_{i+1}, &\textrm{ if } i < 4\\
 -f_1-f_2-f_3, &\textrm{ if } i = 4\\
\end{array}
\right.
$$
If $k=2$ and $\ell$ is any nonnegative integer, the following defines an embedding $\phi: \mathbb Z^{\ell+3} \to \mathbb Z^{\ell+4}$. 
$$\phi(e_i)=\left\{\begin{array}{ll}
f_i- f_{i+1}, &\textrm{ if } i \neq 3\\
-f_1 - f_2 - f_4, &\textrm{ if } i = 3\\
\end{array}
\right.
$$
Finally, if $\ell=2$ and $k$ is any positive integer, the following defines an embedding $\phi: \mathbb Z^{k+3}\to~\mathbb Z^{k+4}$. 
$$\phi(e_i)=\left\{\begin{array}{ll}
f_i- f_{i+1}, &\textrm{ if } i \neq k+1\\
f_{k+1} +f_{k+3} +f_{k+4}, &\textrm{ if } i = k+1\\
\end{array}
\right.
$$

We now show that embeddings do not exist in all other cases.
Suppose $\phi$ is such an embedding.  By Lemma~\ref{lemma1}, we may assume that $\phi(e_i) = f_i-f_{i+1}$ for all $i \neq k+1$. Now suppose  $\phi(e_{k+1})=\Sigma_{j=1}^{k+\ell+2}a_j f_j$.  It follows that
$a_1=a_2=\dots=a_k=a_{k+1}-1$,   and $a_{k+2}+1=a_{k+3}=a_{k+4}=\dots=a_{k+\ell+2}$. We now have
\begin{equation}\label{nophi} -3=\phi(e_{k+1})\cdot \phi(e_{k+1})=-ka^2-(a+1)^2-(b-1)^2-\ell b^2,\end{equation}
where $a=a_1=a_2=\dots=a_k$ and $b=a_{k+3}=a_{k+4}=\dots=a_{k+\ell+2}$. Therefore, if $k\ge4$, then $a=0$ and if $\ell \ge 4$, then $b=0$. If both $k\ge4$ and $\ell \ge 4$, then $a=b=0$ and Equation \eqref{nophi} gives the contradiction $3=2$. In the remaining cases, it is easy to verify that Equation~\eqref{nophi} has no integral solutions in $a$ and $b$. Therefore, $\phi$ cannot exist.
\end{proof}

\begin{proposition}\label{negdefembedding-n-negative}
Let $m>1$, $n>0$ and even, and  $\Gamma_{-}$  be the negative definite Goeritz matrix of $C(m,-n)$, as given in Proposition~\ref{Goeritz matrices}.  Then there is an embedding $\phi$ as given in Theorem~\ref{JK theorems} \eqref{phi I minus embedding} if and only if $(m,-n)=(5,-2)$ or at least one of $m$ or $n$ is 4. 
\end{proposition}

\begin{proof}
We are looking for an embedding $\phi: (\mathbb Z^{m+n-2}, \Gamma_- \oplus [-(mn-1)]) \to (\mathbb Z^{m+n-2}, -Id)$ where $ \Gamma_- \oplus [-(mn-1)]$ is given by
\begin{center}
   \includegraphics[width=8.5cm]{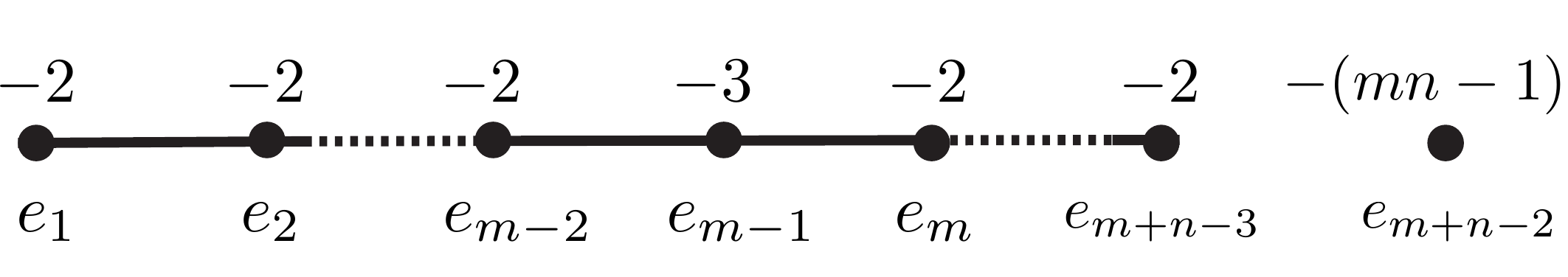}
\end{center}
By Lemma~\ref{lemma2}, we must have one of three possible cases:
\begin{enumerate}
\item $(m,-n)=(5,-2)$
\item $m=4$
\item $n=4$
\end{enumerate}
It remains to show that in each of these cases, the embeddings given in Lemma~\ref{lemma2} can be extended to $e_{m+n-2}$. The reader can check that  this can be done for each case, respectively, by defining $\phi(e_{m+n-2})$ as follows:
\begin{enumerate}
\item $3f_5$
\item $f_1+f_2+f_3 -2f_4-2f_5-\cdots -2f_{n+2}$
\item $-2f_1-2f_2-\dots -2f_{m-1}+f_m+f_{m+1}+f_{m+2}$
\end{enumerate}
\end{proof}

\noindent To summarize the results of these obstructions, consider the following conditions which will apply to a knot $C(m,n)$ where $n$ is either positive or negative. Specifically, conditions (a) and (b) will be used when $n>0$ whereas conditions (c) and (d) will be used when $n<0$.

\begin{enumerate}
\renewcommand{\theenumi}{\alph{enumi}}
\item There exists no integer solutions to $m=n x^2 + 2x +y^2$.
\item There exists no integer solutions to $n=m x^2 + 2x +y^2$.
\item There exists no integral solutions to the set of equations
\begin{align*}
\label{short XXT equations}
m &= x_{11}^2+x_{12}^2+x_{13}^2 \\ 
-n &=  x_{21}^2+x_{22}^2+x_{23}^2 \\
-1& =x_{11} x_{21}+ x_{12} x_{22}+x_{13} x_{23}
\end{align*}
\item $m \neq 4$, $n \neq -4$, and $(m,n)\neq (5,-2)$.
\end{enumerate}
\vspace{.2cm}

Using the results of Yasuhara and Jabuka--Kelly in this section together with the mod 8 values of $\sigma(K) + 4 \mbox{Arf}(K)$ from Proposition~\ref{sigma}, we obtain the following proposition.

\begin{proposition} \label{gamma4>1 conditions} If $m>1$ and  $n$ is even and nonzero, then the following table summarizes  those conditions given above which imply that $\gamma_4(C(m,n)) \ge 2$.
\begin{center}
\begin{tabular}{|c|c|c||c|c|} \hline
 & \multicolumn{2}{|c||}{$n > 0$} & \multicolumn{2}{|c|}{$n<0$} \\ \hline
 (mod 4) & $n \equiv 0$  & $n \equiv 2$ & $n \equiv 0$ & $n \equiv 2$ \\ \hline
 $m \equiv 0$ &   (a) and (b) &  (a) and (b) & (d) & (d) \\ \hline
 $m \equiv 1$  &  (a) and (b) &  (a) &  (d) &  (c) and (d) \\ \hline
 $m \equiv 2$  &  (a) and (b) & always & (d) &  (c)\\ \hline
 $m \equiv 3$  &  (a) and (b) &  (b) &  (d) & always  
 \\ \hline
 \end{tabular}
\end{center}
\end{proposition}

The table in Proposition~\ref{gamma4>1 conditions} is read as follows. For example, if $n<0$, $n \equiv 2 \Mod 4$, $m \equiv 1 \Mod4 $, and conditions (c) and (d) hold, then $\gamma_4(C(m,n)) \ge 2$.

We now list several infinite families of double twist knots with $\gamma_4=2$. In Theorem~\ref{double twist knots with gamma4 equal to 2}, we include at least one infinite family for each of the 16 cases stated in Proposition~\ref{gamma4>1 conditions}. Before stating the theorem, we prove the following lemma that helps us  analyze conditions (a) and (b) above and another condition that will appear in Section~\ref{gamma_4 is 3}. Theorem~\ref{double twist knots with gamma4 equal to 2} is in not meant to be exhaustive. We found many additional infinite families that we do not list here. (See for example, the discussion of twist knots in Section~\ref{twist knots}.)  

\begin{lemma} \label{ab condition} Suppose that $r$, $s$, and $t$ are positive integers.  If $r< s t^2-2t$ and $(x,y,z)$ is an integer solution to $r=s x^2+2x+y^2+z^2$, then $|x|<t$. 
\end{lemma}

\begin{proof} Assume $(x,y,z)$ is an integer solution to $r=s x^2+2x+y^2+z^2$. Because $s$ is positive, the function $s x^2+2x$ is increasing for $x> -1/s$ and decreasing for $x < -1/s$. Thus, if $x \ge t$, then $r = sx^2+2x+y^2+z^2 \ge s t^2 +2t +y^2+z^2 \ge s t^2 -2t  > r$ which is a contradiction. Similarly, if $x \le -t$, then $r = sx^2+2x+y^2+z^2 \ge s t^2 - 2t +y^2+z^2 \ge s t^2 - 2t >r$ which also gives a contradiction. Therefore, $|x|<t$.
\end{proof}  

\begin{theorem}\label{double twist knots with gamma4 equal to 2} If $m>1$, $n$ is even and nonzero,  and $(m,n)$ is one of the following, then $\gamma_4(C(m,n))=2$.
\begin{enumerate}
\item $(2,n)$, with $n>0$ and $n\equiv 2\Mod{4}$,  

\item $(m,2)$, with $m\equiv 2 \Mod{4}$, 

\item $(m,n)$, with $m \equiv 3 \Mod{4}$, $n<0$, and $n\equiv 2 \Mod{4}$,    

\item  $(m,m)$, with $m \equiv 0 \Mod{4}$ and $m$ not a square, 

\item $m$ is odd, $n\equiv 0 \Mod{4}$, $n<0$,  and $n\ne -4$,  

\item $(8,n)$, with $n<0$ and $n\ne -4$,  

\item $(m,-8)$, with $m$ even and $m\ne 4$, 

\item $m \equiv 1\pmod{4}$, $m$ not a square, $n\equiv 2 \pmod{4}$, and $n>m+2$, 

\item $m \equiv 3 \pmod{4}$,  $n>0$, $n\equiv 2 \pmod{4}$, and $m>n+2$, 

\item $(m, m+10)$, with $m>2$  even, $m$ not a square, and  $m+10$ not a square, 

\item $(13+4k,8+4k)$ where $k\ge 0$, $13+4k$ not a square, and $8+4k$ not a square, 

\item $(7+4k,12+4k)$ where $k\ge 0$ and $12+4k$ not a square,

\item\label{twistex} $(2, -(18+100 k))$ where $k\ge 0$, 

\item $(13+36k,-10)$ where $k\ge0$. 

\end{enumerate}  
\end{theorem}

\begin{proof}
First we show that any knot $K$ in the above list has $\gamma_4(K)\leq 2$. By Proposition~\ref{crosscap number of double-twist knots}, we have that $\gamma_3=2$ in every case except $(4), (6), (7)$, and $(10)$. In case $(4)$, a single horizontal band move of the  type illustrated on the left in Figure~\ref{band moves} with $k=0$ and $\epsilon=1$ will take $C(m,m)$ to $C(m+1, m)$. A second band move of the same type takes $C(m+1, m)$  to the slice knot $C(m+2, m)$. In cases $(6)$ and $(7)$, two band moves of the kind shown on the right in Figure~\ref{band moves} can be used to take $K$ to the unknot. Finally, in case $(10)$, two successive band moves of the type shown on the right in Figure~\ref{band moves} change $C(m,m+10)$ to the slice knot $C(m,m+2)$.

It remains to show that the appropriate conditions stated in Proposition~\ref{gamma4>1 conditions} are met in order to conclude that $\gamma_4(K)>1$ and hence equal to 2 for these knots. If $K$ is a knot in (1) or (2), then $m \equiv n \equiv 2 \Mod{4}$ and $n>0$ and so $\gamma_4(K)=2$ by Proposition~\ref{gamma4>1 conditions}. Similarly, if $K$ is a knot in (3), then Proposition~\ref{gamma4>1 conditions} with $m \equiv 3$, $n \equiv 2 \Mod{4}$, and $n<0$ implies $\gamma_4(K)=2$. 
We remark that  (1)  was proven by Feller and Golla (see Proposition 5.8 in \cite{feller2021}).

Now suppose $K$ is a knot in (4). Then Proposition~\ref{gamma4>1 conditions} with $m \equiv n \equiv 0 \Mod{4}$ and $n>0$ will give $\gamma_4(K)=2$ provided conditions (a) and (b) are met. Because $m=n$, this reduces to a single condition that there are no integer solutions to $m = mx^2+2x+y^2$. Because $m>1$, we have $m < 4m-4$. So by Lemma~\ref{ab condition} with $t=2$ and $z=0$, any integer solution $(x,y)$ to $m=mx^2+2x+y^2$ must have $|x| <2$. However there is no solution when $x=0$ because $m$ is not a square. Moreover $x=\pm 1$ would imply that $\mp 2$ is a square. Therefore, there are no integer solutions to $mx^2+mx+y^2$, and so $\gamma_4(K)=2$. 

For a knot $K$ in (5), (6), or (7), condition (d) is met, and so Proposition~\ref{gamma4>1 conditions} implies $\gamma_4(K) =2$.

If $K$ is a knot in (8), then $m \equiv 1 \Mod{4}$, $n \equiv 2 \Mod{4}$, and $n>0$, so it suffices to show condition (a) is met. Because $m < n -2$, Lemma~\ref{ab condition} with $t=1$ and $z=0$ implies that an integral solution must have $|x|<1$. However, a solution with $x=0$ would contradict that $m$ is not a square. Therefore, condition (a) is met, and so $\gamma_4(K) =2$.

If $K$ is a knot in (9), then $m \equiv 3 \Mod{4}$, $n \equiv 2 \Mod{4}$, and $n>0$ so it suffices to show condition (b) is met. Because $n < m -2$, Lemma~\ref{ab condition} with $t=1$ and $z=0$ gives that an integral solution must have $|x|<1$. However, $x=0$ contradicts that $n$ is not a square, and so $\gamma_4(K) =2$.

If $K$ is a knot in (10), then by Proposition~\ref{gamma4>1 conditions} with $m \equiv 0$, $n \equiv 2 \Mod{4}$, and $n>0$ or with $m \equiv 2$, $n\equiv 0 \Mod{4}$, and $n>0$ it suffices to show conditions (a) and (b) are met. Because $m < m +8$, Lemma~\ref{ab condition} with $t=1$ gives that an integral solution to $m=(m+10)x^2+2x+y^2$ must have $|x|<1$. But, there is no solution with $x=0$ because $m$ is not a square. Thus, condition (a) is met. On the other hand, $m>2$ implies that $m+10 < 9m-6$. So Lemma~\ref{ab condition} with $t=3$ gives that any integer solution to $m+10=mx^2+2x+y^2$ must have $|x|<3$. If $x=0$, then $m+10$ is a square which contradicts our assumption. If $x = \pm1$, then $8$ or $12$, respectively, would be a square which is a contradiction. If $x = 2$, then an integer solution would give that $6=3m+y^2$ and this contradicts that $m>2$ is even. Finally, if $x=-2$, then $14=3m+y^2$. This implies that $m=4$, because $m >2$. However, this now gives $2$ is a square which is a contradiction. Therefore (b) is met as well, and so $\gamma_4(K)=2$.

If $K$ is a knot in (11), then by Proposition~\ref{gamma4>1 conditions} with $m \equiv 1$, $n \equiv 0 \Mod{4}$, and $n>0$ it suffices to show conditions (a) and (b) are met. Because $8+4k < 11+4k$, Lemma~\ref{ab condition} with $t=1$ gives that an integral solution to $8+4k=(13+4k)x^2+2x+y^2$ must have $|x|<1$. But, there is no solution with $x=0$ because $8+4k$ is not a square. Thus, condition (b) is met. On the other hand, because $13+4k < 4(8+4k)-4$, Lemma~\ref{ab condition} with $t=2$ gives that any integer solution to $13+4k=(8+4k)x^2+2x+y^2$ must have $|x|<2$. If $x=0$, then $13+4k$ is a square which contradicts our hypothesis. Whereas if $x = \pm1$, then $3$ or $7$, respectively, would be a square which is a contradiction. Therefore (a) is met as well, and so $\gamma_4(K)=2$.

If $K$ is a knot in (12), then by Proposition~\ref{gamma4>1 conditions} with $m \equiv 3$, $n \equiv 0 \Mod{4}$, and $n>0$ it suffices to show conditions (a) and (b) are met. Because $7+4k < 10+4k$, Lemma~\ref{ab condition} with $t=1$ gives that an integral solution to $7+4k=(12+4k)x^2+2x+y^2$ must have $|x|<1$. But, there is no solution with $x=0$ because $7+4k \equiv 3 \Mod{4}$ guarantees that $7+4k$ is not a square. Thus, condition (a) is met. Because $12+4k < 4(7+4k)-4$, Lemma~\ref{ab condition} with $t=2$ gives that any integer solution to $12+4k=(7+4k)x^2+2x+y^2$ must have $|x|<2$. If $x=0$, then $12+4k$ is a square which contradicts our hypothesis. If $x = \pm1$, then $3$ or $7$, respectively, would be a square which is a contradiction. Therefore (b) is met as well, and so $\gamma_4(K)=2$.

For a knot $K$ in (13), it suffices to show that condition (c) is satisfied. Without loss of generality we may assume $x_{11}= \epsilon$, $x_{12}=\delta$, and $x_{13}=0$ where $\epsilon, \delta \in \{-1,1\}$  for the first equation of (c) to be true. It now follows from the third equation that $x_{22}= - \delta (\epsilon x_{21}+1)$ and so the second equation becomes $18+100k = x_{21}^2+(\epsilon x_{21}+1)^2+x_{23}^2$. One can computationally verify  that this equation has no solutions modulo 25. Therefore (c) is met, and so $\gamma_4(K)=2$.

For a knot $K$ in (14), it suffices to show that conditions (c) and (d) are satisfied. Condition (d) is clearly satisfied. From the second equation of (c) we have $10=x_{21}^2+x_{22}^2+x_{23}^2$ and so without loss of generality we may assume $x_{21}=3\epsilon$, $x_{22}=\delta$, and $x_{23}=0$ where $\epsilon, \delta \in \{-1,1\}$. It now follows from the third equation that $x_{12}=-\delta(3\epsilon x_{11}+1)$ and so the second equation becomes $13+36k = x_{11}^2+(3\epsilon x_{11}+1)^2+x_{13}^2$.  One can computationally verify  that this equation has no solutions modulo 9.  Therefore (c) is met as well, and so $\gamma_4(K)=2$.
\end{proof}

As already stated, each of the 14 families in Theorem~\ref{double twist knots with gamma4 equal to 2} is infinite. Notice that in cases (10)--(12), because the distance between consecutive squares grows without bound, there are infinitely many  $m$ and $n$ that satisfy the hypothesis. Alternatively, in cases (11) and (12), if $k \equiv 0 \Mod{4}$, then $13+4k$, $8+4k$, and $12+4k$ are not squares because $13, 8$, and $12$ are not squares modulo $16$.


\section{Double twist knots with $\gamma_4 = 3$}\label{gamma_4 is 3}

In  this section we develop an obstruction to $\gamma_4(C(m,n)) < 3$.  

\begin{proposition}\label{m,n>0 phi II embedding conditions}
Let $m>1$, $n>0$ and even, and  $G_\epsilon$  be the $\epsilon$-definite Goeritz matrix of $C(m,n)$ as given in Proposition~\ref{Goeritz matrices}. Then there is an embedding $\phi$ as given in Theorem~\ref{JK theorems} \eqref{phi II embedding} with $\epsilon=1$ if and only if there exist integers $x, y$ and $z$ such that
\begin{equation*}\label{xyz plus condition}m =n x^2+2x+y^2+z^2.\end{equation*}
There is such an embedding with $\epsilon=-1$ if and only if there exist integers $x, y$ and $z$ such that
\begin{equation*}\label{xyz minus condition}n =m x^2+2x+y^2+z^2.\end{equation*}
Consequently, if $m \equiv n \equiv 2 \Mod{4}$, there exists no integer solutions to $m=n x^2 + 2x +y^2 + z^2$, and there exists no integer solutions to $n=m x^2 + 2x +y^2 + z^2$, then $\gamma_4(C(m,n))\ge 3$.
\end{proposition}
\begin{proof}
The proofs for $\epsilon=1$ or $-1$ are similar. We consider the case of $\epsilon=-1$ and leave the other case to the reader.

Suppose there is an embedding $\phi$ as given in Theorem~\ref{JK theorems} \eqref{phi II embedding} with $\epsilon=-1$. The negative-definite Goeritz matrix $G_-$ is represented by the weighted graph below. 
\begin{center}
   \includegraphics[width=5cm]{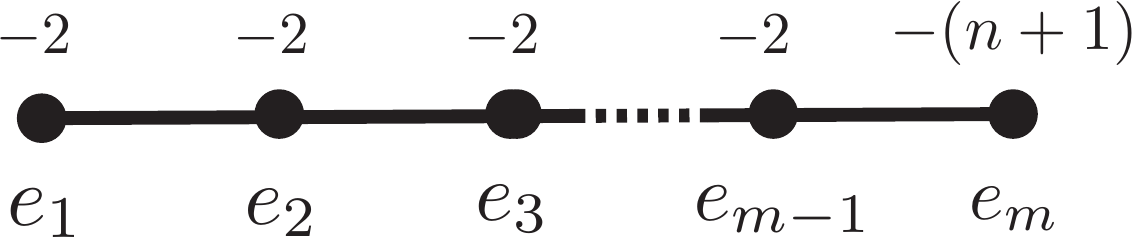}
\end{center}
By Lemma~\ref{lemma1}, we may assume that
$\phi(e_i)=f_i-f_{i+1}\text{ for } 1\le i \le m-1.$ Suppose that $\phi(e_m)=\sum_{i=1}^{m+2} a_i f_i$. Because $e_i \cdot e_m=a_i-a_{i+1} = 0$ for $1\le i\le m-2$, it follows that
$$a_1=a_2=\cdots=a_{m-1}$$ and we rename $a_1=x$. Because $1=e_m \cdot e_{m-1}=-x+a_m$ we have that $a_m=x+1$. Renaming $a_{m+1}$ and $a_{m+2}$ as $y$ and $z$, respectively, we finally obtain $n+1=(m-1)x^2+(x+1)^2+y^2+z^2$ by considering $e_m \cdot e_m$. This is equivalent to $n=m x^2+2x+y^2+z^2$.

Conversely, suppose there exist integers $x,y,z$ such that $n=m x^2+2x+y^2+z^2$. We may define $\phi$ as
 \begin{align*}
      \phi(e_i) &= f_i - f_{i+1} \text{ for } 1\le i\le m-1\\
      \phi(e_m) &=xf_1 + xf_2 + \cdots + xf_{m-1}+ (x+1)f_m + yf_{m+1}+zf_{m+2}.
 \end{align*}
It is easy to check that $\phi$ respects the linking form. Because $G_-$ is negative-definite,  $\phi$ is injective.
Finally, the condition that $m\equiv n \equiv 2 \pmod4$ gives that $\sigma+ 4{\rm Arf}(K)\equiv 4 \pmod 8$, and so there is no embedding of the form required in Theorem~\ref{JK theorems} \eqref{phi II embedding}. Therefore $\gamma_4(C(m,n))\geq 3$.
\end{proof}

We now consider double twist knots with $\gamma_4=3$. (None appear in the Tables~\ref{table1} and \ref{table2}.) As in the previous sections, there are many such families for which Lemma~\ref{ab condition} will guarantee there is no solution to the appropriate equation in Proposition~\ref{m,n>0 phi II embedding conditions}. In the following theorem (which, like Theorem~\ref{double twist knots with gamma4 equal to 2}, is not meant to be exhaustive) we present the proof for one such family.

\begin{theorem}\label{double twist knots with gamma4 equal to 3} 
If $(m,n)=(22+8k,62+8k)$ with $k\ge 0$, then $\gamma_4(C(m,n))=3$.
\end{theorem}

\begin{proof} Recall $\gamma_4(C(m,n))\leq 3$. Because $m \equiv n \equiv 2 \Mod{4}$, Proposition~\ref{m,n>0 phi II embedding conditions} implies that it suffices to show that there are no integer solutions $(x,y,z)$ to either of the following equations 
\begin{align*}
22+8k&=(62+8k)x^2+2x+y^2+z^2\\
62+8k&=(22+8k)x^2+2x+y^2+z^2.
\end{align*} 
First note that $6+8k$ is never the sum of two squares. It it were, then $6$ would be the sum of two squares modulo $8$, which is easy to see is false. Thus neither $22+8k$ nor $62+8k$ are the sum of two squares. By Lemma~\ref{ab condition} with $t=1$, a solution of the first equation would require that $x=0$ but this implies that $22+8k$ is a sum of two squares. On the other hand, $62+8k < (22+8k)2^2-4$ and so by Lemma~\ref{ab condition} with $t=2$, a solution to the second equation would require $|x|<2$. We cannot have $x=0$ because $62+8k$ is not a sum of two squares. Moreover, a solution with $x = \pm 1$ would imply that $38$ or $42$, respectively, would be a sum of two squares which is also a contradiction. Therefore, $\gamma_4(C(m,n))=3$.
\end{proof}

In \cite{Murakami_Yasuhara}, Murakami and Yasuhara noted that $\gamma_4(K) \le 2 g_4(K)+1$ where $g_4$ is the orientable 4-ball genus of a knot. This follows immediately from the fact that a nonorientable spanning surface for $K$ can be obtained by adding a half-twisted band to an oriented spanning surface. Moreover, Murakami and Yasuhara conjectured that there exists a knot $K$ for which the equality is achieved (Conjecture 2.1 of \cite{Murakami_Yasuhara}). Their conjecture was confirmed by Jabuka and Kelly \cite{JK} where they show that the knot $8_{18}$ has $\gamma_4(K)=3$ and $g_4(K)=1$. Because $\gamma_4(K)=0$ when $K$ is slice, if $\gamma_4(K)=2g_4(K)+1$, then we must have $\gamma_4(K) \ge 3$. Among all prime knots of 10 crossings or less, the knot $8_{18}$ is the only one with $\gamma_4=3$ and so this is the only knot with 10 or fewer crossings for which $\gamma_4(K) = 2 g_4(K)+1$. The double twist knots $C(22+8k,62+8k)$ from Theorem~\ref{double twist knots with gamma4 equal to 3} provide an infinite family of  knots which satisfy Murakami and Yasuhara's conjecture.

\begin{corollary}\label{MY conjecture} For all $k \ge 0$, if $K=C(22+8k,62+8k)$, then
$$\gamma_4(K) = 2 g_4(K)+1.$$
\end{corollary}

\begin{proof} By Theorem \ref{double twist knots with gamma4 equal to 3}, we have $\gamma_4(K)=3$. On the other hand, $g(K)=1$ where $g$ is the genus of the knot because $22+8k$ and $62+8k$ are both even. Since $\gamma_4(K) \le 2 g_4(K) + 1 \le 2 g(K)+1=3 $, it follows that $\gamma_4(K) = 2g_4(K)+1$.
\end{proof}

\section{Twist knots}\label{twist knots}

This paper began as an attempt to find $\gamma_4$ for all twist knots, $C(m,2)$, where $m>0$. While this remains an open project, we summarize in Table~\ref{twist knots summary table} those twist knots for which we can determine $\gamma_4$. Because the  crosscap number of $C(m,2)$ is 2, we must have $\gamma_4(C(m,2))\in \{0,1,2\}$. The only twist knot with $\gamma_4=0$ is the slice knot $C(4,2)$ (first proven by Casson and Gordon \cite{CassonGordon}). In Theorem~\ref{double twist knots with gamma4 equal to 1} we found  only five twist knots with $\gamma_4=1$ (of which all but $C(9,2)$ were already known). However, having only explored very specialized band moves to reach a slice 2-bridge knot, we suspect there are many more. Finally,   Table~\ref{twist knots summary table} lists a sampling of infinite families with $\gamma_4=2$. The first two infinite families appear in Theorem~\ref{double twist knots with gamma4 equal to 2}, as items (2) and (9), respectively. Item (2) is given by Feller and Golla~\cite{feller2021}, while item (9) generalizes their result. Proofs that the additional families have $\gamma_4=2$ follow from Proposition~\ref{families of twist knots with gamma4 equal to 2} and Theorem~\ref{values of t and j that work}.

\begin{table}[htp]
\begin{center}
\begin{tabular}{|l|l|}
\hline
$\gamma_4=0$ & $C(4,2)$\\
\hline
$\gamma_4=1$ & $C(1,2), C(3,2), C(5, 2), C(8,2), C(9,2)$\\
\hline
$\gamma_4=2$ & $C(2+4k, 2), k\ge 0$\\
& $C(7+4k, 2), k\ge 0$\\
& $C(4j+4 tk, 2),k\ge 0 \text{ where $t$ and $j$ are given in Table~\ref{tj-values for m congruent to 0 mod 4}}$\\
& $C(1+4j+4tk, 2),  k\ge 0 \text{ where $t$ and $j$ are given in  Table~\ref{tj-values for m congruent to 1 mod 4}}$\\
\hline
\end{tabular}
\end{center}
\caption{Known values of $\gamma_4$ for some twist knots.}
\label{twist knots summary table}
\end{table}

\newpage
\begin{proposition}\label{families of twist knots with gamma4 equal to 2}
Suppose that $j$ and $t$ are integers with $1<j<t$. 
\begin{enumerate}
\item If $2j \not\equiv x^2+x+2 y^2 \Mod{t}$ for all integers $x$ and $y$, then $\gamma_4(C(4j+4tk, 2))=2$ for all $k\ge 0$.
\item  If $2j \not\equiv x^2+x+2 y^2+2y \Mod{t}$ for all integers $x$ and $y$, then $\gamma_4(C(1+4j+4tk, 2))=2$ for all $k\ge 0$.
\end{enumerate}
\end{proposition}
\begin{proof}We begin with (1).
Using Proposition~\ref{gamma4>1 conditions}, it suffices to show that both conditions (a) and (b) hold. Consider first condition (b) and notice that that $2<(4j+4tk)\ell^2-2 \ell$ with $\ell=1$ because $1<j<t$. Thus, by Lemma~\ref{ab condition} any integral solution to the equation $2=(4j+4tk)x^2+2x+y^2$ must have $|x|<1$. But this would imply that $2$ is a square, a contradiction. Next, suppose that the equation $4j+4tk=2x^2+2x+y^2$, or equivalently, $1+4j+4tk=x^2+(x+1)^2+y^2$ has an integral solution. Reducing this modulo $4$ leads to the fact that $y$ must be even. Abusing notation, we replace $y$ with $2y$ and arrive at $2j+2tk=x^2+x+2y^2$, and finally that $2j \equiv x^2+x+2 y^2 \Mod{t}$. Thus, because we are assuming this is false, condition (a) is true.

We now prove (2). Using Proposition~\ref{gamma4>1 conditions}, it suffices to show that  condition (a)  holds. Suppose that the equation $1+4j+4tk=2x^2+2x+y^2$, or equivalently, $2+4j+4tk=x^2+(x+1)^2+y^2$ has an integral solution. Reducing this modulo $4$ leads to the fact that $y$ must be odd. Abusing notation, we replace $y$ with $2y+1$ and arrive at $2j+2tk=x^2+x+2y^2+2y$, and finally that $2j \equiv x^2+x+2 y^2+2 y \Mod{t}$. Thus, because we are assuming this is false, condition (a) is true.
\end{proof}

Here are two examples that illustrate Proposition~\ref{families of twist knots with gamma4 equal to 2}.
 If $t=25$ then the expression $x^2+x+2 y^2 \Mod{25}$ takes on all possible values except $\{1, 11, 16, 21\}$. Halving these numbers modulo $25$ gives the set $\{8,13,18,23\}$ and   four infinite families with $\gamma_4=2$:
\begin{itemize}
\item $C(32+100k, 2), k\ge 0$
\item $C(52+100k, 2), k\ge 0$
\item $C(72+100k, 2), k\ge 0$
\item $C(92+100k, 2), k\ge 0$
\end{itemize}

If $t=25$ then the expression $x^2+x+2 y^2+2y \Mod{25}$ takes on all possible values except $\{3, 8, 13, 23\}$. Halving these numbers modulo $25$ gives the set $\{4, 14, 19, 24\}$ and   four infinite families of the second kind given in Proposition~\ref{families of twist knots with gamma4 equal to 2}: 
\begin{itemize}
\item $C(17+100k, 2), k\ge 0$
\item $C(57+100k, 2), k\ge 0$
\item $C(77+100k, 2), k\ge 0$
\item $C(97+100k, 2), k\ge 0$
\end{itemize}

It is not hard to find values of $t$ and $j$ that satisfy the hypothesis of Proposition~\ref{families of twist knots with gamma4 equal to 2}. In the previous two examples, $t=25$ is the smallest such $t$ for each of the two families given in the Proposition. 
If $T$ is a multiple of $t$, then  each family given in Proposition~\ref{families of twist knots with gamma4 equal to 2} associated with $T$ will be a subset of the family associated with $t$. Thus we seek ``primitive'' values of $t$. The following theorem lists all such values of $t$ up to $1000$, together with the associated values of $j$. We do not know if there are an infinite number of such $t$.

\begin{theorem}\label{values of t and j that work} For the following twist knots $K$, we have  $\gamma_4(K)=2$.  
\begin{enumerate}
\item $K=C(4 j+4 t k, 2)$, where $k\ge 0$   and $(t,j)$ are specified as in Table~\ref{tj-values for m congruent to 0 mod 4}.
\item $K=C(1+4 j+4 t k, 2)$, where $k\ge 0$   and $(t,j)$ are specified as in Table~\ref{tj-values for m congruent to 1 mod 4}.
\end{enumerate}
\end{theorem}
\begin{table}[htp]
\begin{center}
\begin{tabular}{|c||l|}
\hline
$t$ & $j$\\
\hline
25 & 8, 13, 18, 23 \\
\hline
49 & 13, 20, 27, 34, 41, 48 \\
\hline
169 & 8, 34, 47, 60, 73, 86, 99, 112, 125, 138, 151, 164\\
\hline
529 & 20, 43, 89, 
   112, 135, 158, 181, 204, 227, 250, 273, 296, 319,  342, 365,\\ 
   & 388, 
   411, 434, 457, 480, 503, 526\\
\hline
841 & 18, 47, 76, 134, 163, 192, 
   221, 250, 279, 308, 337, 366, 395, 424, 453, \\
   & 482, 511, 540, 569, 
   598, 627, 656, 685, 714, 743, 772, 801, 830\\
   \hline
961 & 27, 58, 89, 
   151, 182, 213, 244, 275, 306, 337, 368, 399, 430, 461, 492,\\
   &
   523, 554, 585, 616, 647, 678, 709, 740, 771, 802, 833, 864, 895, 926, 
   957 \\
\hline

\end{tabular}
\end{center}
\caption{Values of $t$ and $j$ that imply $\gamma_4(C(4j+4tk, 2))=2$ for all $k\ge 0$.}
\label{tj-values for m congruent to 0 mod 4}
\end{table}
   
\begin{table}[htp]
\begin{center} 
\begin{tabular}{|c||l|}
\hline
$t$ & $j$\\
\hline
 25 & 4, 14, 19, 24 \\
 \hline
 49 & 4, 11, 25, 32, 39, 46 \\
 \hline
 169 & 11,24, 37, 50, 76, 89, 102, 115, 128, 141, 154, 167 \\
 \hline
 529 &
   14, 37, 60, 83, 106, 129, 152, 175, 221, 244, 267, 290, 313, 336, 359,\\
   & 382, 405, 428, 451, 474, 497, 520 \\
   \hline
 841 &
   25,54, 83, 112, 141, 170, 199, 228, 257, 286, 344, 373, 402, 431, 460,\\
   & 489, 518, 547, 576, 605, 634, 663, 692, 721, 750, 779, 808, 837 \\
   \hline
 961 &
   19, 50, 81, 112, 143, 174, 205, 236, 267, 298, 329, 391, 422, 453, 484,\\
   & 515, 546, 577, 608, 639, 670, 701, 732, 763, 794, 825, 856, 887, 918, 949\\
   \hline
\end{tabular}
\end{center}
\caption{Values of $t$ and $j$ that imply $\gamma_4(C(1+4j+4tk, 2))=2$ for all $k\ge 0$.}
\label{tj-values for m congruent to 1 mod 4}
\end{table}

\pagebreak


\bibliographystyle{plain}
\bibliography{myReferences}

\begin{table}[htp]
\begin{center}
   \includegraphics[width=5.65in]{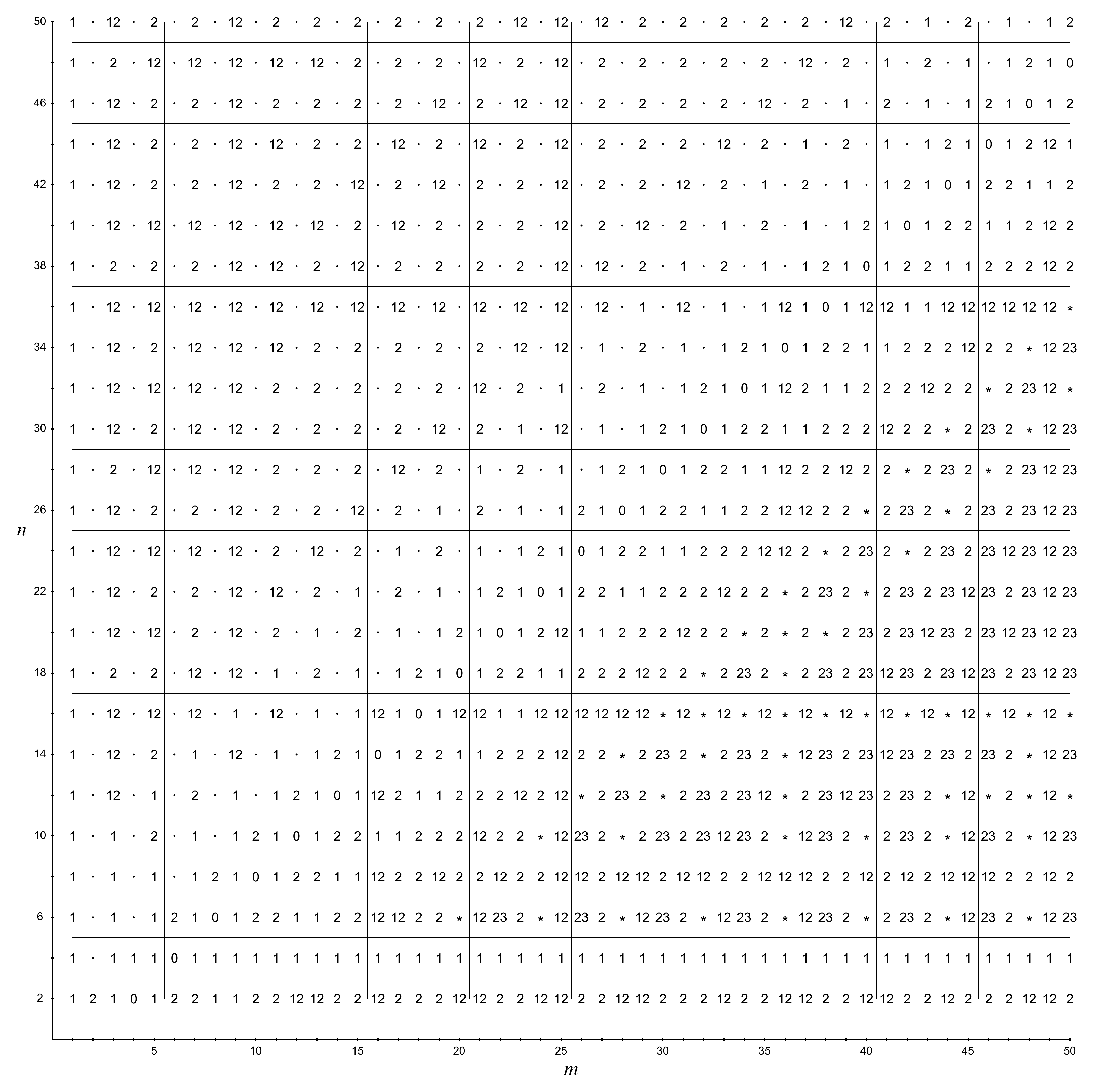}
\end{center}
\caption{The value, or range of values, of $\gamma_4(C(m,n))$ for $m>0$, $n>0$ and even. The symbol 12 denotes $\gamma_4 \in \{1,2\}$, the symbol 23 denotes $\gamma_4 \in \{2,3\}$, and the symbol *  denotes $\gamma_4 \in \{1,2,3\}$. A dot indicates a duplication.}
\label{table1}
\end{table}%

\pagebreak

\begin{table}[htp]

\begin{center}
   \includegraphics[width=5.65in]{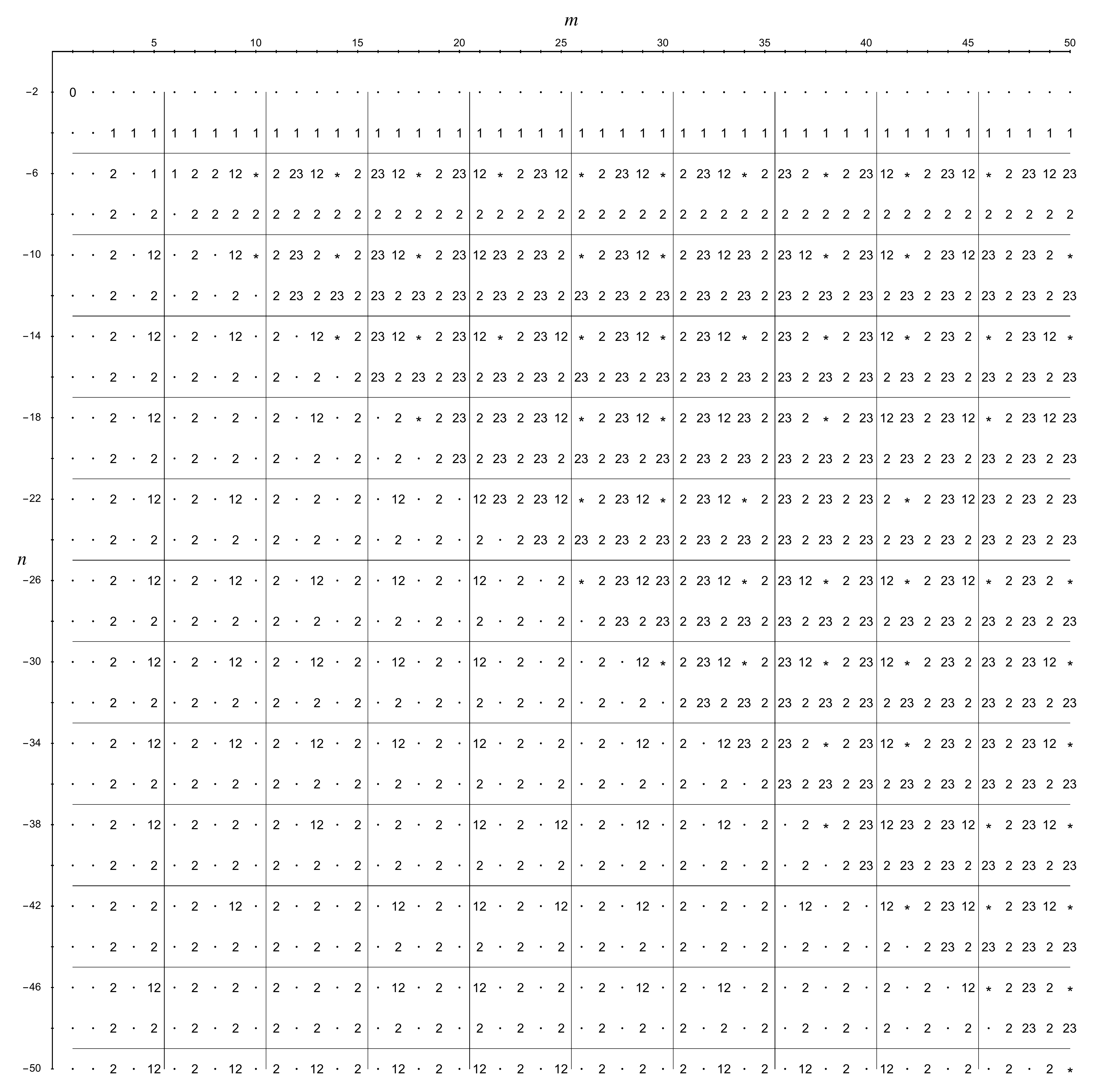}
\end{center}

\caption{The value, or range of values, of $\gamma_4(C(m,n))$ for $m>0$, $n<0$ and even. All symbols are as in Table~\ref{table1}.}
\label{table2}
\end{table}%

\end{document}